\documentclass[11pt]{amsart}
\usepackage{a4wide}
\usepackage{amsmath, amsfonts, amssymb, amscd, graphicx, amsthm, cite}
\usepackage{mathrsfs,url}
\usepackage{hyperref}
\usepackage{color}
\usepackage{comment}
\usepackage{enumitem}
\theoremstyle{plain}

    \newtheorem{thm}{Theorem}
    \newtheorem{theorem}[thm]{Theorem}

    \newtheorem{lemma}[thm]{Lemma}

    \newtheorem{proposition}[thm]{Proposition}

    \newtheorem{corollary}[thm]{Corollary}

    \newtheorem{conjecture}[thm]{Conjecture}
    
\theoremstyle{definition}

    \newtheorem{question}[thm]{Question}

    \newtheorem{definition}[thm]{Definition}

\newcommand{\N}{{\mathbb N}}
\newcommand{\Z}{{\mathbb Z}}

\newcommand{\F}{{\mathbf F}}

\DeclareMathOperator{\id}{id}
 
\DeclareMathOperator{\Aut}{Aut}

\DeclareMathOperator{\End}{End}
\DeclareMathOperator{\Pol}{Pol}

\DeclareMathOperator{\dom}{dom}

\DeclareMathOperator{\Sub}{Sub}

\newcommand{\alg}[1]{\mathbf{#1}}
\newcommand{\clone}[1]{\mathcal{#1}}
\newcommand{\clonoid}[1]{\mathcal{#1}}

\newcommand{\clodC}{\clonoid{C}}
\newcommand{\clodD}{\clonoid{D}}
\newcommand{\clodO}{\clonoid{O}}

\newcommand{\algA}{\alg{A}}
\newcommand{\algB}{\alg{B}}
\newcommand{\algC}{\alg{C}}
\newcommand{\algD}{\alg{D}}
\newcommand{\algL}{\alg{L}}
\newcommand{\algM}{\alg{M}}

\newcommand{\algR}{\alg{R}}
\newcommand{\algS}{\alg{S}}
\newcommand{\algU}{\alg{U}}
\newcommand{\cloA}{\clone{A}}
\newcommand{\cloB}{\clone{B}}
\newcommand{\cloC}{\clone{C}}
\newcommand{\cloO}{\clone{O}}

\renewcommand\vec{\mathbf}
\newcommand{\va}{\vec{a}}
\newcommand{\vb}{\vec{b}}

\newcommand{\vn}{\vec{n}}
\newcommand{\vm}{\vec{m}}
\newcommand{\vr}{\vec{r}}
\newcommand{\vu}{\vec{u}}

\newcommand{\vx}{\vec{x}}
\newcommand{\vy}{\vec{y}}

 \newcommand{\restrict}[1]{\lvert_{#1}}

\DeclareMathOperator{\Clo}{\mathsf{Clo}}

\DeclareMathOperator{\rk}{rk}
\DeclareMathOperator{\GL}{GL}
\DeclareMathOperator{\chars}{char}
\DeclareMathOperator{\supp}{supp}
\DeclareMathOperator{\SMP}{SMP}

\DeclareMathOperator{\Id}{Id}

\newcommand{\Col}[1]{\mathcal{C}(#1)}

\newcommand{\ste}[1]{\textcolor{green}{ #1}}

\newcommand{\mk}[1]{\textcolor{red}{Mk: #1}}

\author[S.~Fioravanti]{Stefano Fioravanti}
\address{Department of Algebra\\ Charles University\\Praha\\ Czechia
}
\email{stefano.fioravanti66@gmail.com}

\author[M.~Kompatscher]{Michael Kompatscher}
\address{Department of Algebra\\ Charles University\\Praha\\ Czechia}
\email{kompatscher@karlin.mff.cuni.cz}

\author[B.~Rossi]{Bernardo Rossi}
\address{Siena, Italy}
\email{bernardo.rossi96@gmail.com}

\thanks{This paper was supported by Charles University under grant number PRIMUS/24/SCI/008. Michael Kompatscher received additional funding from the Czech Science Foundation
(GA\v{C}R grant No. 25-16324S), and Charles University Research Center program No. UNCE/24/SCI/022.}

\title{Clonoids over vector spaces}

\begin{document}

\begin{abstract}
Clonoids are sets of finitary operations between two algebraic structures that are closed under composition with their term operations on both sides. We conjecture that, for finite modules $\algA$ and $\algB$ there are only finitely many clonoids from $\algA$ to $\algB$ if and only if $\algA$, $\algB$ are of coprime order.

We confirm this conjecture for a broad class of modules $\algA$. In particular we show that, if $\algA$ is a finite $k$-dimensional vector space, then every clonoid from $\alg A$ to a coprime module $\algB$ is generated by its $k$-ary functions (and arity $k-1$ does not suffice). In order to prove this results, we investigate `uniform generation by $(\algA,\algB)$-minors', a general criterion, which we show to apply to several other existing classifications results. Based on our analysis, we further prove that the subpower membership problem of certain 2-nilpotent Mal'cev algebras is solvable in polynomial time.
\end{abstract}

\maketitle

\section{Introduction}
A \emph{clone} is a set of finitary operations on a given set that contains all projections and is closed under functional composition. Clones are a fundamental object in universal algebra, as they describe the set of term operations of any algebraic structure. Thus, starting with Post's classification of clones on a two-element set \cite{post-lattice}, there has been a rich history of structural results on finite algebras based on clone theory, see e.g. \cite{szendrei-clones}. On the other hand, clones theory has also several applications in theoretical computer science. One of the most prominent examples is the study of polymorphism clones to determining the complexity of fixed template constraint satisfaction problems (see e.g. the survey \cite{BKW-polymorphisms}), which lead to Bulatov's and Zhuk's independent proofs of the CSP dichotomy theorem \cite{bulatov-dichotomy}, \cite{zhuk-dichotomy-short,zhuk-dichotomy}.

In recent years, generalizations of clones such as \emph{minions} and \emph{clonoids} have received increasing attention. Minions are sets of finitary operations from a set $A$ to a (possibly different) set $B$ that are closed under permutation and identification of variables, as well as the addition of dummy variables. Research on minions gained much traction due to their application on (promise) constraint satisfaction, see e.g. the survey \cite{BBKO-PCSP}. A \emph{clonoid} from an algebra $\algA$ to an algebra $\algB$ is a minion that is closed under (pre-)composition with operations of $\algA$ and (post-)composition with the operations of $\algB$. While the term `clonoid' was first introduced by Aichinger and Mayr in~\cite{AM-eqclasses}, the concept has appeared under various names earlier, e.g. in \cite{ CF-clonoids}. 

Due to their weak structure, it is does not seem to be feasible to explicitly describe all minions between two given finite sets. A first hurdle is their sheer number: it was shown in \cite{pippenger-minions} that there are already continuum many minions between two 2-element sets. However, depending on the expressiveness of the algebras $\algA$ and $\algB$, much more can be said about the clonoids between $\algA$ and $\algB$. Sparks showed in \cite{sparks-clonoids}, that there are only finitely many clonoids from $\algA$ to $\algB$ if $\algB$ has a near-unanimity term, and at most countably many, if $\algB$ has an edge-term. However, also a rich source algebra $\algA$ can lead to finiteness results: As showed in \cite{LS-discriminator}, every clonoids from an algebra $\algA$ with a discriminator term, to an algebra $\algB$ on the same set, gives rise to only finitely many clonoids.

We would also like to highlight a series of papers by Lehtonen and coauthors \cite{LN-minorsclique,LS-discriminator, CL-stability, lehtonen-majorityclonoids, lehtonen-NUclonoids}, which culminated in \cite{lehtonen-Booleanclonoids} to a complete characterization of all two-element algebras $(\algA,\algB)$, for which there are at most countable clonoids, together with a complete classification of clonoids in all such cases.

In this paper, we are interested in classifying clonoids, for which $\algB$ is a module. Such clonoids often appear naturally when studying the term operations of Mal'cev algebras with a central congruence; this connection was recently made more explicit by Peter Mayr's definition of the \emph{difference clonoid}, see e.g. \cite{kompatscher-SMP2nil}. In particular, 2-nilpotent algebras give rise to clonoids between modules $\algA$ and $\algB$.

As an example of an application let us let us mention the classification of extensions of square-free groups in \cite{AM-Zpqextensions} and later \cite{fioravanti-groupexpansions}, which heavily relied on the classification of certain clonoids.

There is already a series of classification results in the literature. The clonoids between groups of prime order $\Z_p$ and $\Z_q$ were completely classified by Kreinecker \cite{kreinecker-zpclonoids} (in the case that $p=q$) and the first author, for $p \neq q$ \cite{fioravanti-clonoids}. Interestingly, in the latter case, there are only finitely many clonoids that are all generated by their unary functions (while for $p=q$ one obtains infinitely many). 

In subsequent generalizations of these results, it was also always observed that there are only finitely many clonoids between certain modules $\algA$ and $\algB$ of coprime order - the most recent, and most general such finiteness result is by Mayr and Wynne \cite{MW-clonoidsmodules}, and applies to all cases where $\algA$ has a distributive lattice of submodules. This leads us to the following Conjecture:

\begin{conjecture} \label{conjecture:main}
Let $\algA$ and $\algB$ be finite modules. Then, there are only finitely many clonoids from $\algA$ to $\algB$, if and only if $\algA$ and $\algB$ are of coprime order.
\end{conjecture}

We remark that the ``only if direction'' follows from \cite[Theorem $1.3$]{MW-clonoidsmodules}, where the authors proved that for modules with a common divisor, one can always construct an infinite ascending chain of clonoids.

 The first main contribution of our paper is to confirm Conjecture \ref{conjecture:main} if $\algA$ is a finite vector space. In the case where $\algB$ is coprime to the vector space $\algA$, we are moreover able to give an explicit description of the lattice of all clonoids (Theorem \ref{theorem:clonoidlattice}). In particular, this answers \cite[Question 1.2]{MW-clonoidsmodules}, which asked for a classification of clonoids, if $\algA$ is the group $(\Z_p)^2$, for some prime $p$. By combining our result with \cite{MW-clonoidsmodules}, we are further able to confirm Conjecture \ref{conjecture:main} for all modules $\algA =\F_1^{k_1}\times \cdots \times \F_n^{k_n} \times \alg D$ as $\F_1\times \cdots \times\F_n \times \algR$-module, where $\alg D$ is a distributive $\algR$-module, and all $\F_i$ are fields.

Our classification is based on proving that, in the case of a vector space $\algA = \F^k$ and coprime $\algB$, the set of all operations from $A$ to $B$ is \emph{uniformly generated} by $k$-ary minors. While this technique was already hinted at in \cite{MW-clonoidsmodules}, we are going to discuss it in great detail in Section \ref{sec:unifgen}, and argue that many of the existing clonoid classifications follow the same pattern. We furthermore show that the arity $k$ is optimal, i.e., not every clonoid from $\F^k$ to $\algB$ is generated by their $k-1$-ary functions. This coincides with a lower bound on the arity of generators of the \emph{full} clonoid between modules that we derive in Section \ref{sec:lowerbound}, and was already pointed out in the Bachelor thesis by Jan Van\v{e}\v{c}ek \cite{vanecek-thesis}.

Last, let us mention that, clonoids between affine algebras were also (sometimes implicitly) used to discuss some computational problems of Mal'cev algebras. The fact that some clonoids between coprime modules are generated by unary functions, which allows for a representation of their elements by nice normal forms, was used in \cite{KKK-CEQV2nil} to prove that checking polynomial identities in a given 2-nilpotent algebra can always be done in polynomial time. In \cite{mayr-VLloop} and \cite{KompatscherMayr2026} such normal forms were similarly used to find finite equational basis for some 2-nilpotent loops.

While we are not going to discuss such syntactic aspects of clonoids, we will apply our results to the \emph{subpower membership problem} of certain algebras. The subpower membership problem (SMP($\algA$)), for a fixed finite algebra $\algA$, is the computational problem whether a given partial operation $f\colon A^n \to A$ can be extended to a term operation of $\algA$. In \cite{IMMVW-subpowers} it was asked, whether the subpower membership problem is always polynomial time solvable for algebras with few subpowers. This question, however remains unsolved even in the Mal'cev case. The easiest examples that are not covered by existing tractability results \cite{mayr-SMP, BMS-SMP}, are 2-nilpotent algebras, for which it was shown in \cite{kompatscher-SMP2nil} that the problem is polynomial time equivalent to a similar `interpolation problem' for their clonoids. In Section \ref{sect:SMP} we are going to use our results, to prove that the subpower membership problem of a big class of 2-nilpotent algebras is in P.

\subsection*{Organization of the paper}

The paper is organized as follows. In Section~\ref{sec:prel} we introduce the necessary concepts from universal algebra and fix some notation. In Section~\ref{sec:unifgen} we develop a theory of `uniform generation' and `uniform representation' by minors, a concept introduced in~\cite{MW-clonoidsmodules}. In particular we provide a purely combinatorial criterion for all $(\algA,\algB)$-clonoids to be generated by their $k$-ary functions (Theorem \ref{theorem:unigen}).

Section~\ref{sec:mainres} presents the main results of the paper, Theorem \ref{theorem:main} and discusses some consequences. In Section \ref{sect:SMP} we apply our results to the subpower membership problem. In Sections~\ref{sec:lowerbound} we provide a lower bound on the arity of the generators of the full clonoid between finite modules, which is sharp in our setting. Section~\ref{sec:conclusion} is dedicated to future work and possible applications of our results.

\section{Background}\label{sec:prel}

In this section, we introduce some necessary notation, key definitions, and foundational results from the literature concerning clone theory, clonoids, and module theory that will serve as the basis for our analysis. For further background in universal algebra we refer the reader to~\cite{BS-universal-algebra}.

Concerning the basic notation, tuples consisting of elements are represented in boldface, with their individual components denoted as, e.g., $\boldsymbol{a}=(a_1,\dots,a_n)$. Furthermore, we write $[n]$ to indicate the set $\{1,\dots,n\}$ and $[n]_0$ for $[n] \cup \{0\}$.

\subsection{Clones and clonoids}\label{sec:cloclon}

\begin{definition}
For two sets $A, B$, we define $\clodO_{A,B} = \bigcup_{n \in \N} B^{A^n}$, that is, $\clodO_{A,B}$ is the set of all finitary operations from $A$ to $B$. Moreover, we are going to write $\clodO_A = \clodO_{A,A}$. For a set of operations $\clodC \subseteq \clodO_{A,B}$ we denote its $n$-ary part by $\clodC^{(n)}$, i.e. $\clodC^{(n)} = \clodC \cap B^{A^n}$.
\end{definition}


\begin{definition}
Let $A$ be a set and $\cloA \subseteq \clodO_A$. 
Then $\cloA$ is a \emph{clone} on $A$ 
if it contains all projections and 
is closed under composition, i.e.
\begin{itemize}
\item $\pi_i^n \in \cloA$, for every $1\leq i \leq n$ and $n \in \N$, where $\pi_i^n(x_1,\ldots,x_n) = x_i$
\item $f \in \cloA^{(n)}, g_1, \ldots, g_n \in \cloA^{(k)} \Rightarrow f \circ (g_1,\ldots,g_n) \in \cloA^{(k)}$ for all $k,n \in \N$.
\end{itemize}
\end{definition}

\begin{definition}
An \emph{algebraic structure} or \emph{algebra} $\algA = (A,(f_i)_{i\in I})$ is a pair consisting of a set $A$ (the \emph{universe} of $\algA$), and a family of finitary functions $f_i \colon A^{k_i} \to A$
(the \emph{basic operations} of $\algA$). The \emph{term clone of $\algA$} is the smallest clone containing all basic operations of $\algA$. We are going to denote it by $\Clo(\algA)$. More generally, for a set of operations $F$ on $A$ we are also going to write $\Clo(F)$ for the smallest clone containing $F$. We say two algebras $\algA$ and $\algA'$ are \emph{term-equivalent} if $\Clo(\algA) = \Clo(\algA')$.
\end{definition}

If not specified otherwise, we are going to use the same letter in different typeset for a clone/algebra and its domain, e.g., $\cloB$ is a clone on set $B$, $\algB$ is an algebra with universe $B$. An operation $m\colon A^3 \to A$ is called a \emph{Mal'cev operation}, if it satisfies the identities $m(y,x,x) \approx m(x,x,y) \approx y$. An operation $v \colon A^k \to A$ for $k \geq 3$ is called a \emph{near unanimity operation}, if it satisfies $v(y,x\ldots,x) \approx v(x,y,x, \ldots,x) \approx \ldots \approx v(x,x, \ldots, x,y) \approx x$.

We next introduce the main notion of this paper:

\begin{definition}[Clonoid]
Let $\cloA$ be a clone on a set $A$, and $\cloB$ be a clone on a set $B$. We then say that $\clodC \subseteq \clodO_{A,B}$ is a \emph{clonoid from $\cloA$ to $\cloB$} (or \emph{$(\cloA,\cloB)$-clonoid}, for short) if, for all $n,k \in \N$
\begin{itemize}
\item $f_1,\ldots,f_n \in \clodC^{(k)}, g\in \cloB^{(n)} \Rightarrow g\circ (f_1,\ldots,f_n) \in \clodC^{(k)}$,
\item $f \in \clodC^{(k)}, g_1,\ldots,g_k \in \cloA^{(n)} \Rightarrow f\circ (g_1,\ldots,g_k) \in \clodC^{(n)}$.
\end{itemize}
If $\cloA = \Clo(\algA)$ and $\cloB = \Clo(\algB)$ for algebras $\algA,\algB$ then we say that $\clodC$ is a clonoid from $\algA$ to $\algB$ or, equivalently, that $\clodC$ is an $(\algA, \algB)$-clonoid. 
\end{definition}

We remark that in the literature also alternative terminology is used; for example \cite{CF-clonoids} describes $(\cloA,\cloB)$-clonoids as \emph{function classes that are as left stable under $\cloB$ and right stable under $\cloA$}. The term \emph{clonoid} was coined by Aichinger and Mayr in \cite{AM-eqclasses}, and initially only applied to the case, in which $\cloA$ is the projection clone. 

Every $(\cloA,\cloB)$-clonoid is also a \emph{minor-closed set}, or \emph{minion} from $A$ to $B$ (in the sense of~\cite{pippenger-minions}, respectively ~\cite{BBKO-PCSP}), i.e., it is closed under composition with projections from the inside. Conversely, minions can be seen as clonoids from the clone of projections on $A$ to the clone of projections on $B$.

\begin{definition}
Let $\cloA$, $\cloB$ be clones on sets $A$ and $B$. For a set of operations $F \subseteq \clodO_{A,B}$ from $A$ to $B$, we define $\langle F \rangle_{\cloA,\cloB}$ to be the \emph{$(\cloA,\cloB)$-clonoid generated by $F$}, i.e., the smallest $(\cloA,\cloB)$-clonoid containing $F$. 
We say that a clonoid $\clodC$ from $\cloA$ to $\cloB$ is \emph{finitely generated} if there is a finite $F \subseteq \clodC$ with $\clodC = \langle F \rangle_{\cloA,\cloB}$.
\end{definition}

It is not hard to see that, for fixed $\cloA$, $\cloB$, the clonoids from $\cloA$ to $\cloB$ form a lattice with join and meet operations $\clodC \land \clodD = \clodC \cap \clodD$ and $\clodC \lor \clodD = \langle \clodC \cup \clodD \rangle_{\cloA,\cloB}$.

We next introduce polymorphism minions, in orderto give a relational description of clonoids:


\begin{definition}
A \emph{relational structure} is a pair $\mathbb{A} = (A, (R_i^{\mathbb{A}})_{i \in I})$ consisting of a set $A$ and a family of relations $R^{\mathbb{A}}_i \subseteq A^{k_i}$ of finite arity $k_i \geq 1$. The family of relational symbols and their arity $((R_i,k_i))_{i\in I}$ is called the language of $\mathbb A$.

Let $\mathbb{A} = (A, (R_i^{\mathbb{A}})_{i \in I})$ and $\mathbb{B} = (B, (R_i^{\mathbb{B}})_{i \in I})$ be relational structures in the same language. A \emph{homomorphism} for $\mathbb A$ to $\mathbb B$ is a map $h\colon A \to B$ that preserves all relations, i.e. $\va = (a_1,\ldots,a_{k_i}) \in R_i^{\mathbb A} \Rightarrow h(\va) = (h(a_1),\ldots,h(a_{k_i})) \in R_i^{\mathbb B}$, for all $i\in I$. An \emph{($n$-ary) polymorphism} from $\mathbb{A}$ to $\mathbb{B}$ is a map $f\colon A^n \to B$, such that for each relation symbol $R_i$, and for all tuples $\va_1, \ldots,\va_n \in R_i^{\mathbb A}$, also $f(\va_1, \ldots,\va_n) \in R_i^{\mathbb B}$; here, $f$ is evaluated coordinate-wise. This is equivalent to $f$ being a homomorphism from the direct power $\mathbb A^n$ to $\mathbb B$.

We denote the set of all polymorphisms from $\mathbb A$ to $\mathbb B$ by $\Pol(\mathbb A, \mathbb B)$. We further write $\Pol(\mathbb A) = \Pol(\mathbb A, \mathbb A)$. It is well-known that $\Pol(\mathbb A)$ forms a clone, which we call the \emph{polymorphism clone} of $\mathbb A$. It is also easy to see that $\Pol(\mathbb A, \mathbb B)$ is a clonoid from $\Pol(\mathbb A)$ to $\Pol(\mathbb B)$. In the literature $\Pol(\mathbb A, \mathbb B)$ is commonly referred to as \emph{polymorphism minion} from $\mathbb A$ to $\mathbb B$.
\end{definition}

In~\cite{pippenger-minions}, Pippenger introduced a Galois-connection between the minions between two finite sets $A$ and $B$, and pairs of relational structures on the other hand. His results in particular imply that every minion between $A$ and $B$ is equal to a polymorphism minion $\Pol(\mathbb A,\mathbb B)$.\footnote{In fact, slightly more generally, \cite[Section 3]{pippenger-minions}, also already gave a relational description the clonoids between essentially unary clones $\cloA$ to the projection clone on $B$.} In~\cite{CF-clonoids}, Couceiro and Foldes generalized Pippenger's Galois connection twofold - both to consider also sets of arbitrary cardinality, and relations that are invariant under given clones. For our paper, mainly the following consequence of \cite[Theorem 3]{CF-clonoids} is relevant:

\begin{theorem} \label{theorem:CF}
Let $\cloA, \cloB$ be clones on two finite sets $A$ and $B$. Then $\clodC \subseteq \clodO_{A,B}$ is a clonoid from $\cloA$ to $\cloB$ if and only if there are relational structures $\mathbb A$, $\mathbb B$ with  $\clodC = \Pol(\mathbb A, \mathbb B)$, and $\cloA \subseteq \Pol(\mathbb A)$ and $\cloB \subseteq \Pol(\mathbb B)$.
\end{theorem}

\subsection{Modules and vector spaces}
Next, we discuss some notation for modules and vector spaces. Rings in this paper are always assumed to be unital, i.e. every ring $\algR = (R,+,-,0,\cdot, 1)$ contains a multiplicative identity $1$. All modules in this paper are going to be \emph{left} modules. More precisely, an $\algR$-module is a one-sorted algebra $\algA = (A,+,0,-,(r)_{r \in \algR})$, where $(A,+,0,-)$ is the underlying abelian group, and every ring element is identified with the unary function $r(x) = rx$.

Note that for an $\mathbf R$-module $\algA$, $\Clo(\algA)$ consists exactly of the functions given by linear combinations over $R$, i.e., $t(x_1,\ldots,x_n) = \sum_{i=1}^n r_i x_i$ with $r_1,\ldots,r_n \in R$. Using matrix multiplication, we will also write $t(\vx) = \vr^T \vx$ for short, where $\vr = (r_1,\ldots,r_n)^T \in R^n$, and $\vx = (x_1,\ldots,x_n)^T \in A^n$. More generally, note that a function $t\colon A^n \to A^m$ is in $t \in (\Clo(\mathbf A)^{(n)})^m$ if and only if $t(\vx) = A \vx$, where $A \in R^{m\times n}$.

If $I = \{ r\in R \mid \forall a \in A \colon r(a) = 0 \}$ denotes annihilator of an $\algR$-module $\algA$ and $R'=R/I$, then $\algA$ is term equivalent to the module $\algA' = (A,+,0,-,(r)_{r \in \algR'})$, with $(r+I)(a) = r(a)$. Thus, every module is term equivalent to a faithful module.

If $\mathbf F$ is a field, then, following the convention for modules, $\mathbf A = (F^k, +,0,-,(r)_{r \in F})$ is the $k$-dimensional vector space over $\mathbf F$. We are going to slightly abuse notation, and also use the notation $\algA = \F^k$ (although $\algA$ is not a direct power of $\F$ in the universal-algebraic sense). We consider the elements of $\algA$ as $k$-dimensional \emph{row} vectors $\vx^T = (x_1,\ldots,x_k)$ with $x_1,\ldots,x_k \in F$. We further identify tuples $(\vx_1^T,\ldots,\vx_n^T) \in (A)^n$ with the matrix $X \in \mathbf F^{n \times k}$, whose $i$-th row is equal to $\vx_i^T$. We choose this notational convention, as it is consistent with what we already established for modules. In particular, every term operation $t \in \Clo(\mathbf A)^{(n)}$, given by a linear combination $t(\vx_1^T,\ldots,\vx_n^T) = \sum_{i=1}^n a_i \vx_i^T$ can be compactly written as $t(X) = \va^T X$. Moreover, an operation $t\colon A^n \to A^m$ is in $t \in (\Clo(\mathbf A)^{(n)})^m$ if and only if $t(X) = AX$ for a matrix $A \in F^{m\times n}$.

We next recall some standard definitions from linear algebra:

\begin{definition}
For a ring $\algR$ and a matrix $X \in R^{k \times l}$, the rank $\rk(X)$ of $A$ is defines as the minimal $n \in \N$ such that $X = AB$ for some $A \in  R^{k \times n}$, $U \in R^{n \times l}$. If $n = \rk(X)$, we are going to refer to any decomposition $X = AB$ with $A \in  R^{k \times n}$, $U \in R^{n \times l}$ as \emph{rank factorization} of $X$.
\end{definition}

In the case of matrices over a field, $\rk(X)$ is equal to all the usual definition of $\rk(X)$ as column or row rank (see e.g.~\cite[Section 5.4]{BR-linearalgebra}). Moreover, the following well-known facts hold:

\begin{lemma}\label{lemma:the wellknown remark}
Let $\F$ be a field, let $k,l\in\N$, let $X \in F^{k \times l}$ and let $n$ be the rank of $X$. If $X = AU = A'U'$ are two distinct rank factorizations of $X$ then there exists an invertible matrix 
$T \in F^{n\times n}$, 
with $A' = AT$ and $U' = T^{-1}U$.
\end{lemma}

For a matrix $X \in \F^{k \times l}$ over a field $\F$ we let $C(X) = \{ X\va \mid \va \in F^l \}$ denote the \emph{columnspace of $X$}.


For any linear subspace $V\leq \F^k$, we are going to write $V^{\bot} = \{\vy \in \F^k \mid \forall \vx \in V: \vy^T \vx = 0 \}$ for the space of vectors `orthogonal' to $V$. Note that $(V^{\bot})^{\bot} = V$, $\dim(V) + \dim(V^\bot) = k$ and $(V\cap W)^\bot = V^{\bot} + W^{\bot}$ hold. However (unlike in an actual inner product space) $V$ and $V^{\bot}$ need not be complementary.

\section{Uniformly generated functions}\label{sec:unifgen}

In this section, we discuss sets of operations that are \emph{uniformly generated by their $n$-ary minors}, a notion that was first introduced in~\cite{MW-clonoidsmodules}. In particular, we are going to provide a general criterion for all clonoids between clones $\cloA$ and $\cloB$ to be generated by their $n$-ary part (Corollary \ref{corollary:unigen}). We then discuss how this criterion is preserved under some basic constructions, and show how it can be used to reprove several known results.


\begin{definition}
Let $\cloA$ be a clone on a set $A$. For $k,l,n \in \N$, let $R_{n}^{k,l}(\cloA)$ be the set of all the operations $r \in (\cloA^{(k)})^l$ such that there exist $v \in (\cloA^{(n)})^l$ and $w \in (\cloA^{(k)})^n$ such that 
$$r(\vx) = \begin{bmatrix} r_1(\vx)\\ \vdots \\ r_l(\vx) \end{bmatrix} = \begin{bmatrix} v_1 (w_1(\vx), \ldots, w_n(\vx))\\ \vdots \\ v_l (w_1(\vx), \ldots, w_n(\vx)) \end{bmatrix} = v \circ w (\vx).$$
We further define $R_n(\cloA) = \bigcup_{l,k} R_{n}^{k,l}(\cloA)$. For an algebra $\algA$, we will also use the notation $R^{k,l}_n(\algA) = R^{k,l}_n(\Clo(\algA))$ and $R_n(\algA) = R_n(\Clo(\algA))$.
\end{definition}

In words, $R^{k,l}_n(\cloA)$ consists of all operations in $(\cloA^{(k)})^l$ that factor through $\cloA^{(n)}$. As it was already pointed out in~\cite{MW-clonoidsmodules}, it follows straightforwardly from the definition that, for an $\algR$-module $\algA$:
$$R_{n}^{k,l}(\algA) = \{ f \colon A^k \to  A^l \mid \exists M \in R^{l \times k} : \rk(M) \leq n, f(\vx) = M \vx\}.$$
So $R_{n}(\algA)$ corresponds to the matrices over $\algR$ of rank at most $n$. Thus, the following lemma may be considered as a generalization of the fact that $\rk(MN) \leq \min(\rk(M),\rk(N))$ for matrices $M$,$N$ over a given ring:
\begin{lemma} \label{lemma:rkprod}
Let $\cloA$ be a clone. If $s \in R_n^{j,k}(\cloA)$,$r \in R_m^{k,l}(\cloA)$ then $r \circ s \in R_{\min (n,m)}^{j,l}(\cloA)$.
\end{lemma}

\begin{proof}
By definition, there are $u \in (\cloA^{(k)})^m, t \in (\cloA^{(m)})^l$ such that $r = t \circ u$, and there are $w \in (\cloA^{(j)})^n, v \in (\cloA^{(n)})^k$ such that $s = v\circ w$. Thus $r \circ s = t \circ (u \circ v \circ w)$. Since $t \in (\cloA^{(m)})^l$, and $(u \circ v \circ w) \in (\cloA^{(j)})^m$, we get that $r \circ s \in R_{m}^{j,l}(\cloA)$. On the other hand $r \circ s = (t \circ u \circ v) \circ w$ similarly implies that $r \circ s \in R_{n}^{j,l}(\cloA)$. Thus $r \circ s \in R_{n}^{j,l}(\cloA) \cap R_{m}^{j,l}(\cloA) = R_{\min(n,m)}^{j,l}(\cloA)$.
\end{proof}

By the following observation, the operation in $R_n(\cloA)$ are helpful when trying to characterize the clonoids generated by their $n$-ary elements:

\begin{lemma} \label{lemma:ngen}
Let $\cloA, \cloB$ be clones, $f \colon A^l \to B$, $g \colon A^k \to B$, and $\clodC = \langle f \rangle_{\cloA, \cloB}$. Then $g \in \langle \clodC^{(n)}\rangle_{\cloA, \cloB}$ if and only if there are $m \in \N$, $s \in \cloB^{(m)}$, and $r_1\ldots,r_m \in R_{n}^{k,l}(\cloA)$ such that
\begin{equation}\label{primaequazionelemma18}
    g = s \circ (f \circ r_1, \ldots, f \circ r_m).
\end{equation}
In particular, $\clodC = \langle \clodC^{(n)} \rangle_{\cloA, \cloB}$ if and only if there are $m \in \N, s \in \cloB^{(m)}$ and $r_1\ldots,r_m \in R_{n}^{k,k}(\cloA)$ such that
\begin{equation}f = s \circ (f \circ r_1, \ldots, f \circ r_m). \label{eq:ngen}\end{equation}
\end{lemma}

\begin{proof}
Note that $g \in \langle \clodC^{(n)} \rangle_{\cloA, \cloB}$ if and only if it is in the $(\cloA, \cloB)$-clonoid generated by all functions of the form $f \circ (v_1,\ldots,v_k)$, for $v_i \in \cloA^{(n)}$. If we further compose with operations from $\cloA^{(k)}$ from the inside, we obtain all operations of the form $f \circ r$ for $r\in R_n^{k,l}(\cloA)$. Thus $g \in \langle \clodC^{(n)} \rangle_{\cloA, \cloB}$ if and only if there are 
$s \in \cloB^{(m)}$ and $r_1\ldots,r_m \in R_{n}^{k,l}(\cloA)$ with $g = s \circ (f \circ r_1, \ldots, f \circ r_m)$.

For the second statement, note that $\clodC = \langle \clodC^{(n)} \rangle_{\cloA, \cloB}$ if and only if $f \in \langle \clodC^{(n)} \rangle_{\cloA, \cloB}$. The result then straightforwardly follows from \eqref{primaequazionelemma18}.
\end{proof}

Note that it is enough to state $r_i \in R_n(\cloA)$ in both statements of Lemma \ref{lemma:ngen}, as e.g. $r_i \in R_{n}^{k,l}(\cloA)$ is implied by the arities of $f$ and $g$ and the equation \eqref{primaequazionelemma18}. Thus often, when clear from the context, we are going to omit these superscripts.

If the formula \eqref{eq:ngen} holds independently of the choice of $f\colon A^k \to B$ then we say that $B^{A^k}$ is \emph{uniformly generated} by $n$-ary minors. More precisely, following \cite{MW-clonoidsmodules}, we define:

\begin{definition}\label{def:unifgen}
Let $\cloA$, $\cloB$ be clones. We say that $U \subseteq B^{A^k}$ is \emph{uniformly generated by its $n$-ary $(\cloA, \cloB)$-minors}, if there are $s \in \cloB^{(m)}$ and $r_1\ldots,r_m \in R_{n}^{k,k}(\cloA)$ such that for every $f \in U$:
\begin{equation} \label{eq:gen} f = s \circ (f \circ r_1, \ldots, f \circ r_m).
\end{equation}
More generally, for $U \subseteq \clodO_{A,B}$ we say that $U$ is \emph{uniformly generated by its $n$-ary $(\cloA, \cloB)$-minors} if $U^{(k)}$ is uniformly generated by its $n$-ary $(\cloA, \cloB)$-minors, for every $k \in \N$. 
For short, we will also write that $U$ is $(n,\cloA,\cloB)$-UG.
\end{definition}

\begin{definition} \label{def:unifrep}
Let $\cloA$, $\cloB$ be clones. We say that a (partial) operation $I \colon B^{A^k} \to B^{A^l}$ \emph{can be uniformly represented by $n$-ary $(\cloA, \cloB)$-minors}, if there are $s \in \cloB^{(m)}$ and $r_1\ldots,r_m \in R_{n}^{l,k}(\cloA)$ such that for all $f$ in the domain of $I$
\begin{align} \label{eq:gen2} I(f) = s \circ (f \circ r_1, \ldots, f \circ r_m).
\end{align}
If $I$ can be uniformly represented by $n$-ary $(\cloA, \cloB)$-minors for some $n$, we simply say that $I$ \emph{can be uniformly represented (over $(\cloA, \cloB)$)}. For short, we will also write that $I$ is $(\cloA,\cloB)$-UR (or $(n,\cloA,\cloB)$-UR), if $I$ can be represented by ($n$-ary) $(\cloA, \cloB)$-minors.
\end{definition}

\begin{definition} \label{def:unirepformula}
If we consider term clones $\cloA = \Clo(\algA)$ and $\cloB = \Clo(\algB)$ in Definitions \ref{def:unifgen} and \ref{def:unifrep} we will also refer to \emph{uniform generation/representation by $n$-ary $(\algA,\algB)$-minors}. If an operation $I \colon \clodO_{A,B}^{(l)} \to \clodO_{A,B}^{(k)}$ an be uniformly represented by $n$-ary $(\algA, \algB)$-minors, this means that there are term operations $s^{\algB}$, and $r_1^{\algA}\ldots,r_m^{\algA} \in R_{n}^{k,l}(\algA)$, such that 
\begin{align} \label{eq:gen3} I(f) = s^{\algB} \circ (f \circ r_1^{\algA}, \ldots, f \circ r_m^{\algA}).
\end{align}
If, instead we consider the $m$-ary term $s$ (in the language of $\algB$) and $r_1, \ldots, r_m$ (as $k$-tuples of $l$-ary) terms (in the language of $\algA$), then we are going to refer to the (many sorted) term
\begin{align} \label{eq:gen4} I(f) = s \circ (f \circ r_1, \ldots, f \circ r_m),
\end{align}
as a \emph{representation (formula)} of the operation $I$ (by $n$-ary $(\algA, \algB)$-minors).
\end{definition}

In practice, we are sometimes not going to distinguish between an operation and its representation, but this should never cause any confusion.

Let us observe some basic facts:

\begin{lemma} \label{lemma:urobs}
Let $\cloA, \cloB$ be clones. Then
\begin{enumerate}
\item $U \subseteq \clodO_{A,B}^{(k)}$ is $(n,\cloA,\cloB)$-UG iff $\id\restrict U$ is $(n,\cloA,\cloB)$-UR.
\item If an operation $I$ is $(n,\cloA,\cloB)$-UR and $m \geq n$, then $I$ is $m$-ur
\item Assume $I_1,\ldots,I_r \colon \clodO_{A,B}^{(k)} \to \clodO_{A,B}^{(l)}$ are $(n,\cloA,\cloB)$-UR, and $s \in \cloB^{(r)}$.\\ Then $I(f) = s \circ (I_1(f),\ldots,I_i(f))$ is $(n,\cloA,\cloB)$-UR.
\item Assume $I_1\colon \clodO_{A,B}^{(k_1)} \to \clodO_{A,B}^{(k_2)}$ is $(n_1,\cloA,\cloB)$-UR and $I_2\colon \clodO_{A,B}^{(k_2)} \to \clodO_{A,B}^{(k_3)}$ is $(n_2,\cloA,\cloB)$-UR\\ Then $I_2\circ I_1\colon \clodO_{A,B}^{(k_1)} \to \clodO_{A,B}^{(k_3)}$ is $(n,\cloA,\cloB)$-UR, for $n = \min(n_1,n_2)$.
\end{enumerate}
\end{lemma}

\begin{proof}
(1) directly follow from the definitions. (2) straightforwardly follows from the observation that $R_n(\cloA) \subseteq R_m(\cloA)$ for $n \leq m$. To see (3), for every $i \in [r]$ let us fix a uniform representation of $I_i$, i.e. some $g_i \in \cloB^{(m_i)}$, $r_{i,j_i} \in R_n(\cloA)$ for all $j_i \in [m_i]$ such that
\begin{equation*}
  I_i(f) = g_i \circ (f \circ r_{i,1},\ldots, f \circ r_{i,m_i}),   
\end{equation*}
for all $f \in \dom(I_i)$. Let $g = s \circ (g_1,\dots,g_n) \in \cloB$. Then
\begin{equation*}
  I(f) = g \circ (f \circ r_{1,1},\ldots, f \circ r_{1,m_1}, \ldots f \circ r_{r,1},\ldots, f \circ r_{r,m_r})   
\end{equation*}
holds for all $f \in \bigcap_{i \in [r]} \dom(I_i)$. This formula witnesses that $I$ can be uniformly represented by $n$-ary ($\cloA$, $\cloB$)-minors.

Last we prove (4). Let us take uniform representations of $I_1$, $I_2$, i.e. functions $g_1,g_2 \in\cloB$, $r_1,\ldots,r_m \in R_{n_1}(\cloA)$, and $s_1,\ldots,s_l \in R_{n_2}(\cloA)$ such that
\begin{align*}
I_1(f) &= g_1 \circ (f \circ r_1,\ldots, f \circ r_{m}) \\
I_2(h) &= g_2 \circ (h \circ s_1,\ldots, f \circ s_{l}) 
\end{align*}
for all $f \in \dom(I_1)$ and $h \in \dom(I_2)$. Then, for $f \in \dom(I_1) \cap I_1^{-1}(\dom(I_2))$ we obtain:
\begin{align*}
I_2 \circ I_1(f) &= g \circ (f \circ r_1 \circ s_1, f \circ r_2 \circ s_1 \ldots, f \circ r_{m-1} \circ s_{l}, f \circ r_m \circ s_l), \text{ where } \\
g(x_1,x_2,\ldots,x_{mn}) &= g_2(g_1(x_1,\ldots,x_m),\ldots, g_1(x_{(l-1)m+1},\ldots, x_{lm})).
\end{align*}
Lemma \ref{lemma:rkprod} implies that $r_i\circ s_j \in R_n(\cloA)$, for all $i \in [m]$ and $j \in [l]$, thus $I_2 \circ I_1$ can be represented by $n$-ary minors.
\end{proof}

We are ready to prove one of the main theorems of the section which enables to reduce the problem of uniformly generate the $k$-ary part of an ($\cloA$,$\cloB$)-clonoid by $n$-ary minors, with $k>n$, to uniformly generate its $(n+1)$-ary part by $n$-ary minors.

\begin{theorem} \label{theorem:unigen}
Let $\cloA$, $\cloB$ be clones, $n \geq 1$, and let $\cloA_{const} = \Clo(\cloA \cup \{a\}_{a\in A})$ be the clone generated by $\cloA$ and all constant operations over its set. Let $\clodC$ a clonoid from $\cloA_{const}$
to $\cloB$. Then the following are equivalent:
\begin{enumerate}
\item\label{ite:th_unigen1} $\clodC^{(n+1)}$ is $(n,\cloA,\cloB)$-UG
\item\label{ite:th_unigen2} $\clodC^{(k)}$ is $(n,\cloA,\cloB)$-UG for some $k>n$
\item\label{ite:th_unigen3} $\clodC$ is $(n,\cloA,\cloB)$-UG
\item\label{ite:th_unigen4} $\forall k,l \in \N$, every partial operation $I \colon \clodC^{(k)} \to \clodC^{(l)}$ 
that is $(\cloA,\cloB)$-UR is $(n,\cloA,\cloB)$-UR.
\end{enumerate}
\end{theorem}

We would like to stress that, although we assume that $\clodC$ in closed under composition with constant operations from $A$, the equivalence of items (1)-(4) in Theorem \ref{theorem:unigen} holds for $(\cloA, \cloB)$-minors, not just $(\cloA_{const}, \cloB)$-minors. 

\begin{proof}[Proof of Theorem \ref{theorem:unigen}]
The implications (\ref{ite:th_unigen4}) $\to$ (\ref{ite:th_unigen3}) $\to$ (\ref{ite:th_unigen2}) $\to$ (\ref{ite:th_unigen1}) follow straightforwardly from the observations in Lemma \ref{lemma:urobs}. 

To prove (\ref{ite:th_unigen3}) $\to$ (\ref{ite:th_unigen4}) let $I \colon \clodC^{(k)} \to \clodC^{(l)}$ be an operation that can be uniformly representable by $s$-ary $(\cloA,\cloB)$-minors for some $s  \in \N$, and let us assume that $\clodC$ is uniformly generated by $n$-ary $(\cloA,\cloB)$-minors. Lemma \ref{lemma:urobs} (1) implies that $\id \colon \clodC^{(k)} \to \clodC^{(k)}$ can be uniformly representable by $n$-ary $(\cloA,\cloB)$-minors. By Lemma \ref{lemma:urobs} (4) applied to $I = I \circ \id$, we get that $I$ can be uniformly representable by $n'$-ary $(\cloA,\cloB)$-minors, for $n' = \min(n,s)$. By Lemma \ref{lemma:urobs} (2) this finishes the proof. 

It only remains to prove the implication (\ref{ite:th_unigen1}) $\to$ (\ref{ite:th_unigen3}). So let us assume that (\ref{ite:th_unigen1}) holds. 

We first prove that any operation $I\colon \clodC^{(l)} \to \clodC^{(k)}$ that is $(n+1,\cloA,\cloB)$-UR is, in fact, already $(n,\cloA,\cloB)$-UR. Note that, by Lemma \ref{lemma:urobs} (1) it is enough to prove this claim for operations of the form $I_r(f) = f \circ r$, where $r\in R_{n+1}^{k,l}(\cloA)$. For any $r\in R_{n+1}^{k,l}(\cloA)$, by definition, there are $u \in (\cloA^{(n+1)})^{l}, v \in (\cloA^{(k)})^{n+1}$, such that $r = u \circ v$. Thus, we can write $I_r$ as the composition $I_r = I_v \circ \id \circ I_u$, where $I_u(f) = f \circ u$, $I_v(g) = g \circ u$, and $\id$ denotes the identity on $\cloC^{(n+1)}$. By (1) and Lemma \ref{lemma:urobs} (3), $\id$ can be represented by $n$-ary $(\cloA,\cloB)$-minors. By Lemma \ref{lemma:urobs} (2) and (4), also $I_r$ be represented by $n$-ary $(\cloA,\cloB)$-minors, so we are done.

In order to prove (3), we are going to show by induction on $k=1,2,\ldots$, that $\clodC^{(k)}$ is $(n,\cloA,\cloB)$-UG. Note that for $k \leq n+1$ this follows from (1) and Lemma \ref{lemma:urobs} (2). For an induction step $k \to k+1$, by the induction hypothesis, there are operations $s \in \cloB^{(m)}$ and $r_1,\ldots,r_m \in R^{{k,k}}_n(\cloA)$ such that $f  = s \circ (f \circ r_1,\ldots, f \circ r_{m})$, for every $f \in \clodC^{(k)}$. If we fix the last argument of a function $g \in \clodC^{(k+1)}$ to some constant $a_{k+1} \in A$, then $\mathbf x \mapsto g(\mathbf x, a_{k+1})$ is in $\clodC^{(k)}$, since $\clodC$ is closed under composition with $\cloA_{const}$. It follows, that for every $g \in \clodC^{(k+1)}$:
\begin{equation}
\label{eq:unigen}
g(\mathbf x, x_{k+1}) = s( g(r_1(\mathbf x), x_{k+1}), \ldots, g(r_m(\mathbf x), x_{k+1})).
\end{equation}
Note that, for every $r \in R^{k,k}_n(\algA)$, the map $(\mathbf x,x_{k+1}) \mapsto (r(\mathbf x),x_{k+1})$ is in $R^{k+1,k+1}_{n+1}(\algA)$. Thus, equation \eqref{eq:unigen} shows that the identity on $\cloC^{(k)}$ is $(n+1,\cloA,\cloB)$-UR. By our previous observation, it must also be $(n,\cloA,\cloB)$-UR. In other words, $\cloC^{(k)}$ is uniformly generated by $n$-ary $(\cloA,\cloB)$-minors, so we are done.

\end{proof}

As direct corollary of Theorem \ref{theorem:unigen} we obtain the following corollary:

\begin{corollary} \label{corollary:unigen1}
Let $\cloA$ and $\cloB$ clones on sets $A$ and $B$, and let $\clodC$ be a clonoid from $\cloA_{const}$ to $\cloB$ such that $\clodC^{(n+1)}$ is uniformly generated by $n$-ary $(\cloA, \cloB)$-minors. Then $\clodD = \langle \clodD^{(n)} \rangle_{\cloA, \cloB}$ for every subclonoid $\clodD \subseteq \clodC$. 
\end{corollary}

\begin{proof}
By Theorem~\ref{theorem:unigen} (\ref{ite:th_unigen3}) we have that $\mathcal C^{(k)}$ is uniformly generated by $n$-ary $(\cloA, \cloB)$-minors for every $k \in \N$. By Lemma \ref{lemma:ngen}, for every $f \in \clodC$, the clonoid $\langle f \rangle_{\cloA, \cloB}$ is generated by its $n$-ary functions. This straightforwardly implies that every subclonoid $\clodD$ of $\clodC$ is generated by $\clodD^{(n)}$.
\end{proof} 

By applying Corollary \ref{corollary:unigen1} to $\clodC = \clodO_{A,B}$ we obtain the following criterion for all $(\cloA, \cloB)$-clonoids to be generated by their $n$-ary part:

\begin{corollary} \label{corollary:unigen}
Let $\cloA$ and $\cloB$ clones on sets $A$ and $B$. If $\cloO_{A,B}^{(n+1)}$ is uniformly generated by $n$-ary $(\cloA, \cloB)$-minors, then $\clodC = \langle \clodC^{(n)} \rangle_{\cloA, \cloB}$ for every $(\cloA, \cloB)$-clonoid.
\end{corollary}

If $\cloA$ and $\cloB$ are clones on \emph{finite} sets $A$ and $B$ and $\cloO_{A,B}^{(n+1)}$ is uniformly generated by $n$-ary $(\cloA, \cloB)$-minors, then Corollary~\ref{corollary:unigen} implies that there are only finitely many $(\cloA, \cloB)$-clonoids. Classifying all of them reduces to the purely combinatorial problem of determining all subsets of $\cloO_{A,B}^{(n)}$ that are stable under composition with $\cloA$ from the right, and $\cloB$ from the left. This strategy was (implicitly) already used in several classification results in the literature. In the following subsections, we are first going to discuss some other nice consequences of uniform generation, and then review the existing literature under the viewpoint of uniform generation.

\subsection{Basic properties}
For two clones $\cloA_1$, $\cloA_2$ on sets $A_1$, respectively $A_2$, let us define their \emph{direct product} $\cloA = \cloA_1 \times \cloA_2$ as the clone on the set $A=A_1\times A_2$, such that, if we identify $A^n $ with $A_1^{n} \times A_2^{n}$, then $\cloA^{(n)} = \cloA_1^{(n)} \times \cloA_2^{(n)}$, acting on $A^n$ in the natural way. Direct products are compatible with uniform generation in the following sense:

\begin{proposition} \label{proposition:ugproduct}
Let $\cloA_1$, $\cloA_2$, $\cloB$ be clones (on sets $A_1,A_2,B$), and let us assume that $\clodO_{A_i,B}$ is uniformly generated by $n$-ary ($\cloA_i$,$\cloB$)-minors for $i=1,2$. 
Then $\clodO_{A_1\times A_2,B}$ is uniformly generated by $n$-ary ($\cloA_1\times \cloA_2$,$\cloB$)-minors.
\end{proposition}

\begin{proof}
Let $k = n+1$, and let us identify every tuple $\mathbf x \in (A_1\times A_2)^k$ with $(\mathbf x_1,\mathbf x_2) \in A_1^k \times A_2^k$, where $\mathbf x_1$ (respectively $\mathbf x_2$) denotes the coordinate-wise projection to $A^k_1$ (respectively $A^k_2$). Using this notation, we can also consider every $k$-ary operation $f \colon A^k \to B$ as a function $f \colon A_1^k \times A_2^k \to B$.

Since $\clodO_{A_1,B}^{(k)}$ and $\clodO_{A_2,B}^{(k)}$ are uniformly generated by $n$-ary minors, there are $r_1,\ldots r_m \in R_n^{k,k}(\cloA_1)$ and $s_1,\ldots s_l \in R_n^{k,k}(\cloA_2)$ and $g_1,g_2 \in \cloB$ such that for all $f \colon A^k \to B$ and all $(\mathbf{x}_1,\mathbf{x}_2) \in A_1^k \times A_2^k$,
\begin{align*}
f(\mathbf x_1,\mathbf x_2) = &g_1(f(r_1(\mathbf x_1),\mathbf x_2),\dots, f(r_m(\mathbf x_1),\mathbf x_2),\\
f(\mathbf x_1,\mathbf x_2) = &g_2(f(\mathbf x_1,s_1(\mathbf x_2)),\dots, f(\mathbf x_1,s_l(\mathbf x_2)).
\end{align*}
Combining these two equations, we obtain
\begin{equation} \label{eq:ugproduct}
\begin{split}
f(\mathbf x_1,\mathbf x_2) = g_1( &g_2(f(r_1(\mathbf x_1),s_1(\mathbf x_2)),\ldots,f(r_1(\mathbf x_1),s_l(\mathbf x_2))) ,\\
 &g_2(f(r_2(\mathbf x_1),s_1(\mathbf x_2)),\ldots,f(r_2(\mathbf x_1),
s_l(\mathbf x_2))),\ldots,\\
&g_2(f(r_m(\mathbf x_1),s_1(\mathbf x_2)),\ldots, f(r_m(\mathbf x_1),s_l(\mathbf x_2)))),
\end{split}
\end{equation}


for all $f \colon A^k \to B$ and $(\mathbf{x}_1,\mathbf{x}_2) \in A_1^k \times A_2^k$. Since $\cloA^{(k)} 
= \cloA_1^{(k)} \times \cloA_2^{(k)}$, for every pair $r_i \in R_n^{k,k}(\cloA_1)$, $s_j \in R_n^{k,k}(\cloA_2)$ there is a $t_{ij} \in R_n^{k,k}(\cloA)$ with $t_{i,j}(\mathbf x_1,\mathbf x_2) = (r_i(\mathbf x_1),s_j(\mathbf x_2))$. Thus \eqref{eq:ugproduct} witnesses that $\clodO_{A,B}$ is $(n,\cloA,\cloB)$-UG, which finishes the proof.
\end{proof}

We further would like to remark that, if $\cloA,\cloB$ are clones on finite sets, and $\clodO_{A,B}$ is $n$-UG, then every $(\cloA,\cloB)$-clonoid can be described as the polymorphism minion of relational structures of maximal arity $\leq |A|^n$. In fact, we do not require \emph{uniform} generation for this. A proof can be derived from the proof of Theorem \ref{theorem:CF} in~\cite{CF-clonoids}. We include it for completeness:

\begin{lemma} \label{lemma:fingen}
Let $\cloA, \cloB$ be clones on finite sets $A$ and $B$ and $n \in \N$. Then, let us define the relational structures $\mathbb A = (A,\cloA^{(n)})$, and, for every clonoid $\cloC$,  $\mathbb B_{\cloC} = (B,\cloC^{(n)})$ (here, we regard $\cloA^{(n)} \subseteq A^{A^n}$, and $\cloC^{(n)} \subseteq B^{A^n}$ as $A^n$-ary relation on $A$, respectively $B$). Then the following are equivalent:
\begin{enumerate}
\item $\forall (\cloA,\cloB)$-clonoid $\clodC$: $\clodC = \langle \clodC^{(n)} \rangle_{\cloA,\cloB}$,
\item $\forall (\cloA,\cloB)$-clonoid $\clodC$: $\clodC = \Pol(\mathbb A,\mathbb B_{\cloC})$.
\end{enumerate}
\end{lemma}

\begin{proof}
First, observe that for every $(\cloA,\cloB)$-clonoid $\clodC$ the inclusions $\langle \clodC^{(n)} \rangle \subseteq \clodC \subseteq \Pol(\mathbb A,\mathbb B_{\cloC})$ hold. Furthermore, $\mathbb A$ is clearly invariant under $\cloA$; since $\cloC$ is stable under composition with $\cloB$, also $\mathbb B_{\cloC}$ is invariant under $\cloB$. Thus by Theorem \ref{theorem:CF}, $\Pol(\mathbb A,\mathbb B_{\cloC})$ is a $(\cloA, \cloB)$-clonoid.

We next show that $\clodC^{(n)} = \Pol(\mathbb A,\mathbb B_{\cloC})^{(n)}$. By the above, we only need to show the inclusion $\supseteq$. So let $f \in \Pol(\mathbb A,\mathbb B_{\cloC})^{(n)}$. Note that $\pi_1^n,\ldots,\pi_n^n \in \cloA^{(n)}$, where $\pi_i^n$ denotes the $n$-ary projection to the $i$-th coordinate on $A$. Since $f \in \Pol(\mathbb A,\mathbb B_{\cloC})^{(n)}$, we get $ f(\pi_1^n,\ldots,\pi_n^n) \in \clodC^{(n)}$. In other words $f \in \clodC^{(n)}$, which 
which is what we wanted to prove.

In order to prove the statement of the lemma, let us first assume that (1) holds. This implies in particular that every $(\cloA,\cloB)$-clonoid is completely determined by its $n$-ary part. Hence $\cloC^{(n)} = \Pol(\mathbb A,\mathbb B_{\cloC})^{(n)}$ implies that $\cloC = \Pol(\mathbb A,\mathbb B_{\cloC})$ for every $(\cloA,\cloB)$-clonoid $\clodC$. Thus (2) holds.

Conversely assume that (2) holds. By its definition $\mathbb B_{\cloC} = \mathbb B_{\langle \cloC^{(n)} \rangle_{\cloA,\cloB}}$ holds, for every $(\cloA,\cloB)$-clonoid $\cloC$. Now (2) implies that $\clodC = \Pol(\mathbb A,\mathbb B_{\cloC}) = \Pol(\mathbb A,\mathbb B_{\langle \cloC^{(n)} \rangle_{\cloA,\cloB}}) = \langle \cloC^{(n)} \rangle_{\cloA,\cloB}$, for every $(\cloA,\cloB)$-clonoid $\cloC$. Hence (1) holds.
\end{proof}

As a straightforward corollary of Lemma \ref{lemma:fingen} and Corollary~\ref{corollary:unigen}, we get:

\begin{corollary} \label{corollary:fingen}
Let $\cloA, \cloB$ be clones on finite sets $A$ and $B$, and assume $\clodO_{A,B}^{(n+1)}$ is $(n,\cloA, \cloB)$-UG. Then, every $(\cloA, \cloB)$-clonoid $\clodC$ is the polymorphism minion of relational structures of maximal arity $\leq |A|^n$.
\end{corollary}

\subsection{Uniform generation for clonoids between affine clones}

In this section, we are going to review some results from the literature on clonoids between affine clones under the light of uniform generation. 


First, we would like to point out that all of the known classifications of clonoids between coprime modules are (implicitly) based on the proof that $\clodO_{A,B}$ is uniformly generated by $n$-ary minors. 

Arguably the first such result was proven in ~\cite{fioravanti-clonoids}, where the first author classified all the clonoids between 1-dimensional vector spaces over coprime, finite fields. Although not explicitly stated in~\cite{fioravanti-clonoids}, one can easily derive the following statement from it:

\begin{lemma} \label{lemma:onedim}
Let $\algA$ be a one-dimensional vector space over a finite field $\F$, and let $\algB$ be a $\algS$-module, such that $\chars(\F)$ is invertible in $\algS$. Then $\clodO_{A,B}$ is uniformly generated by unary $(\algA,\algB)$-minors.
\end{lemma}

\begin{proof}
A proof of this lemma can be extracted from the proof of \cite[Theorem 4.3.]{fioravanti-clonoids}. Although \cite[Theorem 4.3.]{fioravanti-clonoids} is only stated for the case in which $\algB$ is also a one-dimensional vector space, it can easily be seen that the only required assumption on the module $\algB$ is that the $|F|$, or equivalently $\chars(\F)$, is invertible in $\algS$.
\end{proof}

Lemma \ref{lemma:onedim} together with Proposition \ref{proposition:ugproduct} implies that the statement of Lemma \ref{lemma:onedim} further holds, if $\algA$ is a direct product of finite fields, considered as a regular module. This was in part already observed in \cite{fioravanti-clonoidsproducts}. Using this observation, we can obtain the following slight generalization of \cite[Theorem 1.5.]{MW-clonoidsmodules}:

\begin{theorem} \label{theorem:unarymodules}
Let $\algA$ be a finite, faithful $\algR$-module over a commutative ring $\algR$, and let $\algB$ be a faithful $\algS$-module. Then the following are equivalent:
\begin{enumerate}
    \item $\clodO_{A,B}$ is uniformly generated by unary $(\algA,\algB)$-minors.
    \item Every $(\algA,\algB)$-clonoid is generated by unary functions.
    \item $\chars(R)$ is invertible in $\algS$, $\algR$ is a direct product of finite fields and $\algA$ is isomorphic to the regular $R$-module.
\end{enumerate}
\end{theorem}

\begin{proof}
The implication (1)$\to$(2) holds by Corollary \ref{corollary:unigen}. The implication (3)$\to$(1) follows from Lemma \ref{lemma:onedim} and Proposition \ref{proposition:ugproduct}. The implication (2)$\to$(3) was shown in \cite[Theorem 1.5.]{MW-clonoidsmodules} for the case that $\algB$ is finite. However, it is not hard to see that the exact same proof also works for general $\algB$.
\end{proof}

More generally, in \cite{MW-clonoidsmodules} the following was shown:

\begin{theorem} \label{theorem:distributive}
Let $\algA$ be a finite distributive $\algR$-module (i.e. its lattice of submodules is distributive), and let $\algB$ be an $\algS$-module, such that $\chars(R)$ is invertible in $\algS$. Then $\cloO_{A,B}$ is uniformly generated by  $n$-ary $(\algA,\algB)$-minors, where $n$ is the nilpotence degree of the Jacobson radical of $\algR$.
\end{theorem}

\begin{proof}
This is proved in \cite[Theorem 1.4.]{MW-clonoidsmodules}, although the theorem does not explicitly mention \emph{uniform} generation. A careful reading of \cite{MW-clonoidsmodules} also shows, that, due to Proposition \ref{proposition:ugproduct}, it is enough to prove the statement for uni-serial modules $\algA$.
\end{proof}

Results for modules $\algA$ can be straightforwardly lifted to algebras that are polynomially equivalent to $\algA$, i.e., algebras that, together with all constant operations, generate the same clone as $\algA$. For some special cases, this was already observed in~\cite[Theorem 4.1.]{MW-clonoidsmodules}.

\begin{lemma}\label{lem:unifgenforpolyequiv}
Let $\algA = (A,+,0,-,(r)_{r\in \algR})$ be an $\algR$-module, and let $\algA'$ polynomially equivalent to $\alg A$ and let $\algB$ be an arbitrary algebra.
\begin{enumerate}
\item\label{ite:lem:unifgenforpolyequiv1} If $\clodO_{A,B}$ is uniformly generated by $n$-ary $(\algA, \algB)$-minors, then it is also uniformly generated by $(n+1)$-ary $(\algA', \algB)$-minors.
\item\label{ite:lem:unifgenforpolyequiv2} If $\clodO_{B,A}$ is uniformly generated by $n$-ary $(\algB, \algA)$-minors then it is also uniformly generated by $n$-ary $(\algB,\algA')$-minors.
\end{enumerate}
\end{lemma}

\begin{proof}
It is a folklore result in commutator theory that $\Clo(\algA')$ must contain the ternary operation $x-y+z$ (the unique Mal'cev term in $\Clo(\algA)$) and all binary operations of the form $rx+ (1-r)y$ for $r\in \algR$. For short, let us write $x+^y z = x-y+z$ and $r^y(x) = rx+ (1-r)y$. If we substitute $y$ by some constant $a\in A$ we obtain the algebra $\algA^a = (A,+^a,a,-^a,(r^a)_{r\in \algR})$ that is isomorphic to $\algA$ via the isomorphism $x \mapsto x-a$.

To every $l$-ary term operation $t(x_1,\ldots,x_l) = \sum_{i=1}^{l} r_i x_i \in \Clo(\algA)$ we can similarly associate the $l+1$-ary operation $t^y(x_1,\ldots,x_{l},y) = r_1^y x_1 +^y r_2^y x_2 +^y \ldots +^y r_{l}^y x_{l} \in \Clo(\algA')$, and, for fixed $a\in A$, the $l$-ary operation $t^a(x_1,\ldots,x_{l}) = r_1^a x_1 +^a r_2^a x_2 +^a \ldots +^a r_{l}^a x_{l} \in \Clo(\algA^a)$. In particular, note that $r \in R_n^{k,k}(\algA)$ implies that $r^a \in R_{n}^{k,k}(\algA^a)$ and $r^y \in R_{n+1}^{k+1,k}(\algA')$.

We first show (\ref{ite:lem:unifgenforpolyequiv1}). Let $k > n+1$. By assumption, there exist $g \in \Clo(\algB)^{(l)}$ and $r_1,\ldots,r_{l} \in R_n^{k,k}(\algA)$ such that for all $f\colon A^k \to B$: $f = g \circ (f\circ r_1, \ldots,f\circ r_{l})$. Since $\algA$ and $\algA^a$ are isomorphic for every $a\in A$, also $f = g \circ (f\circ r_1^a, \ldots,f\circ r_l^a)$ holds for all $f\colon A^k \to B$ and every $a\in A$. Alternatively, we can also write this fact as the equation $f(\mathbf x) = g \circ (f(r_1^{y}(\vx)), \ldots,f(r_l^{y}(\vx))$ for all $f\colon A^k \to B$, where we regard $y$ as a new variable. If we substitute $y$ by $x_k$ (the last coordinate of $\vx$), we obtain the equation
$$f(\mathbf x) = g \circ (f(r_1^{x_k}(\vx)), \ldots,f(r_l^{x_k}(\vx))),$$
for all $f\colon A^k \to B$ and $\vx \in A^k$. It is easy to see that $r_i^{y}(\vx) \in R_{n+1}^{k+1,k}(\algA')$ implies that $r_i^{x_k}(\vx) \in R_{n+1}^{k,k}(\algA')$. This finishes the proof.

To see (\ref{ite:lem:unifgenforpolyequiv2}), let $k>n$ and let $g \in \Clo(\algA)$, $r_1,\ldots,r_l \in R_n(\algB)$ such that the formula $f = g \circ (f\circ r_1, \ldots,f\circ r_l)$ witnesses that $\clodO_{B,A}^{(k)}$ is uniformly generated by $n$-ary $(\algB, \algA)$-minors. Then, again by isomorphism, $f = g^{a} \circ (f\circ r_1, \ldots,f\circ r_l)$ holds for every $a \in A$. This in turn implies $f(\mathbf x) = g^{y} (f(r_1(\mathbf x)), \ldots,f(r_l(\mathbf x)))$, for any choice of $y$. If we substitute $y$ by any expression $f(q(\vx))$, such that $q \in R_n(\algA)$ (e.g. $y = f(x_1,x_1,\ldots,x_1)$), we obtain a formula that witnesses that $\clodO_{B,A}^{(k)}$ is uniform generation by $n$-ary $(\algB, \algA')$-minors.
\end{proof}

\subsection{Examples beyond affine clones}
Also Spark's finiteness result in the case in which the target $\cloB$ has a near unanimity operation~\cite{sparks-clonoids} can be derived using the criterion in Corollary \ref{corollary:unigen}. We include a proof of this fact as the proof in~\cite{sparks-clonoids} uses a different (relational) argument.

\begin{lemma}
Let $\cloA$ and $\cloB$ be clones on finite sets $A$, $B$ such that $\cloB$ contains a near unanimity operation of arity $l\geq 3$. Then $\clodO_{A,B}$ is uniformly generated by $|A|^l$-ary $(\cloA,\cloB)$-minors
\end{lemma}

\begin{proof}
Without loss of generality we can assume that $\cloA$ is the projection clone on $A$. Let $m = |A|$ and $n = |A|^l$. By Theorem \ref{theorem:unigen}, it is enough to prove that for some $k>n$, $\clodO_{A,B}^{(k)}$ is uniformly generated by $n$-ary minors. 

For integers $a,b\in \N$, a $b$-ary tuple $\vx$ and a map $\alpha \colon [a] \to [b]$, let $\vx^{\alpha}$ denote the $a$-ary tuple defined by $\vx^{\alpha} = (x_{\alpha(1)},x_{\alpha(2)},\ldots, x_{\alpha(a)})$. Note that $r \in R_n^{k,k}(\cloA)$, if there is an $\alpha \colon [k] \to [k]$ whose image has at most $n$ elements, such that $r(\vx) = \vx^{\alpha}$.

We claim that for every set of functions $S \subseteq [m]^{[k]}$ there is an operation $I_S\colon \clodO_{A,B}^{(k)}\rightarrow\clodO_{A,B}^{(k)}$ that is uniformly representable by $n$-ary $(\cloA,\cloB)$-minors, such that $I_S(f)(\va^{\alpha}) = f(\va^{\alpha})$, for all $\alpha \in S$ and $\va \in A^{k}$. We prove this claim by induction on $|S|$.

First, let us assume that $|S| \leq l$. Let $\beta \colon [k] \to [k]$ be a map that maps every equivalence class of $\bigcap_{\alpha \in S} \ker(\alpha)$ to the same representative, e.g., $\beta(i) = \min(\{i' \in [k] \mid  \forall \alpha \in S \colon \alpha(i') = \alpha(i) \})$. Since $|S| \leq l$, and $\ker(\alpha)$ has at most $|A|$-many equivalence classes for every $\alpha \in S$, we get that $\beta$ attains at most $\leq |A|^l = n$ many values. Then, clearly $I_S(f)(\mathbf x) = f(\vx^\beta)$ satisfies our claim.

For an induction step, let $|S|>l$. Let us fix $l$ distinct elements $\alpha_1,\alpha_2,\ldots, \alpha_l \in S$, and let $S_i = S \setminus \{\alpha_i \}$. By induction assumption, for every set $S_i$ there exist operations $I_{S_i}$ satisfying the claim.

We define $I_S = v\circ (I_{S_1},\ldots, I_{S_l})$, where $v$ is the $l$-ary near unanimity operation of $\cloB$. By Lemma \ref{lemma:urobs}, $I_S$ is uniformly representable by $n$-ary $(\cloA,\cloB)$-minors. For a map $\alpha \in S$ and a tuple $\va \in A^m$, note that the equation $I_{S_{i}}(f)(\va^{\alpha}) = f(\va^{\alpha})$ holds for all but at most one index $i$. Since $v$ is a near unanimity operation, we get $I_S(f)(\va^{\alpha}) = f(\va^{\alpha})$. This finishes the proof of our claim.

Last, observe that every tuple $\va \in A^k$ is of the form $\vb^\alpha$ for some $\alpha \in [m]^{[k]}$, $\vb \in A^m$: to verify this, simply take some enumeration $A = \{b_1,b_2,\ldots,b_m\}$ of $A$, set $\vb = (b_1,b_2,\ldots,b_m)$, and $\alpha(j) = i$ if $a_j = b_i$. Hence, for $S = [m]^{[k]}$ we conclude that $I_S(f) = f$ for all $f \in \clodO_{A,B}^{(k)}$. Since $I_S$ had a uniform representation by $n$-ary minor, we are done.
\end{proof}

If the uniform generation requires the target clone $\cloB$ to contain functions that are not essentially unary (i.e., that depend on more than one argument).

\begin{lemma} \label{lemma:nonUG}
Let $\cloA$ and $\cloB$ be clones on finite sets $A$, $B$ with at least 2 elements and assume that $\clodO_{A,B}$ is uniformly generated by $n$-ary $(\cloA,\cloB)$-minors. Then $\cloB \neq \Clo(\cloB^{(1)})$. 
\end{lemma}

\begin{proof}
For contradiction, let assume that $\cloB = \Clo(\cloB^{(1)})$ and $\clodO_{A,B}$ is uniformly generated by $n$-ary $(\cloA,\cloB)$-minors. Thus, for every $k > n$ there are $s \in \cloB^{(1)}$ and $r\in R_n^{k,k}(\cloA)$, such that $f = s \circ f \circ r$ for all $f \colon A^k\to B$. For every pair $\va_1 \neq \va_2$ in $A^k$ we can find a function $f$ with $s(f(r(\va_1))) = f(\va_1) \neq f(\va_2) = s(f(r(\va_2)))$. It follows that $r(\va_1) \neq r(\va_2)$ for all $\va_1 \neq \va_2$, i.e., $r$ is injective. However, for $k>n$, no element of $R_n^{k,k}(\cloA)$ is injective - contradiction.
\end{proof}

Lehtonen and Szendrei showed in~\cite{LS-discriminator} that every clonoid from the discriminator clone to the projection clone on the same finite set is finitely generated. By Lemma \ref{lemma:nonUG} this finiteness result cannot be proven using the criterion in Corollary \ref{corollary:unigen}.

\subsection{Interpolation by uniformly representable operations}

A helpful step in proving that $\mathcal O_{A,B}^{(k)}$ is uniformly generated by $n$-ary minors may be to first construct $n$-UR operation $I$, such that $I(f)(\vx)$ agrees with $f(\vx)$ for all $f$ and all tuples $\vx \in K$ from some `nice' subset $K \subseteq A^k$. 

An instance of this can be found in \cite[Lemma 2.14 (3)]{MW-clonoidsmodules}, where it is shown that for modules $\algA,\algB$, uniform generation from submodules of $\algA$ to $\algB$ can be lifted to such an `interpolating' operation $I$ that works on the union of said submodules. In the following, we generalize this result to the setting that only requires the target algebra to have a Mal'cev term:

\begin{lemma} \label{lemma:ugsubalgebras}
Let $\algA$, $\algB$ be algebras, such that $\algB$ has a Mal'cev term. Let $\algC_1,\algC_2,\ldots,\algC_l \leq \algA$ be subalgebras such that for every $i \in [l]$, $\clodO_{C_i,B}$ is uniformly generated by $n$-ary minors. Then, there exists a $(n,\algA,\algB)$-UR operation $I$ on $\clodO_{A,B}^{(n+1)}$ with $I(f)(\mathbf x) = f(\mathbf x)$ for $\mathbf x\in \bigcup_{i \in [l]} C_i^{n+1}$.
\end{lemma}

\begin{proof}
Let $m$ be the Mal'cev term of $\algB$, and let $k = n+1$. Our assumption implies, that for every $i \in [l]$ there are $(n,\algA,\algB)$-UR operations $I_i\colon \clodO_{A,B}^{(k)} \to \clodO_{A,B}^{(k)}$ such that $I_i(f)(\vx) = f(\vx)$ if $\vx \in C_i^{k}$. 

Let us inductively define the operations $J_1(f) = I_1(f)$ and $$J_{i+1}(f) = m(J_i(f)(\vx),I_{i+1}(J_i(f))(\vx),I_{i+1}(f)(\vx))$$ for every $i < l$. We claim that, for every $i\in [l]$: $J_i(f)(\vx) = f(\vx)$ for all tuples $\vx \in C_1^k\cup  C_2^k \cup \cdots \cup  C_i^k$. We prove this claim by induction on $i$.

For $i=1$ the statement clearly holds. So let us consider an induction step $i \to i+1$. First, let $\vx \in C_j^k$, for $j \leq i$. By induction assumption, we know that $J_i(f)(\vx) = f(\vx)$. Furthermore, note that, for every $r \in R_n(\algA)$, we get $r(\vx) \in C_j^k$, since $C_j$ is a subalgebra of $\algA$. Let $I_{i+1}(f)(\vx) = s(f(r_1(\vx)),\ldots,f(r_l(\vx)))$ be a representation of $I_{i+1}$. Then, for $\vx \in C_j^k$:
\begin{align*}
I_{i+1}(J_i(f))(\vx) &=  s(J_i(f)(r_1(\vx)),\ldots,J_i(f)(r_l(\vx))\\
&=  s(f(r_1(\vx)),\ldots,f(r_l(\vx)) \\
&= I_{i+1}(f)(\vx).
\end{align*}
By the Mal'cev identities we obtain $J_{i+1}(f)(\vx) = m(f(\vx),I_{i+1}(f)(\vx),I_{i+1}(f)(\vx)) = f(\vx)$.

Next, assume that $\vx \in C_{i+1}^k$. Then, by assumption, $I_{i+1}(J_i(f))(\vx) = J_i(f)(\vx)$ holds, and therefore  $J_{i+1}(f)(\vx) = m(J_i(f)(\vx),J_i(f)(\vx),f(\vx)) = f(\vx)$. We set $I = J_l$. By Lemma \ref{lemma:urobs} (3) and (4) $I$ can be uniformly represented by $n$-ary minors, thus we are done.
\end{proof}

We next show, how `uniform interpolation' is also invariant under endomorphisms of the algebra $R_n^{k,k}(\algA)$, i.e., maps $e \colon A^k \to A^k$ such that $e(r(\vx)) = r(e(\vx))$ holds for every $r \in R_n^{k,k}(\algA)$. Note that the algebra $(A^k, R_n^{k,k}(\algA))$ is a (unary) term reduct of the $k$-th matrix-power $\algA^{[k]}$ of $\algA$ (in the sense of \cite[Exercise 3.12 (4)]{HM-TCT}), so, this includes all endomorphisms of $\algA^{[k]}$. If $\algA = \F^l$ is a vector space, then all maps $X \mapsto XT$, for $X \in \F^{k\times l}$ and $T \in \F^{l\times l}$ are endomorphisms of $R_n^{k,k}(\algA)$.

\begin{lemma} \label{lemma:endomorphisms}
Let $\algA$ and $\algB$ be arbitrary algebras, let $\va \in A^n$, and assume that $I\colon \clodO_{A,B}^{(k)} \to \clodO_{A,B}^{(k)}$ is $(n,\algA,\algB)$-UR such that $I(f)(\va) = f(\va)$, for all $f \in \clodO_{A,B}^{(k)}$. Then, for every $e \in \End( R_n^{k,k}(\algA))$, also $I(f)(e(\va)) = f(e(\va))$ holds.
\end{lemma}

\begin{proof}
Let $I(f) = s\circ (f\circ r_1, \ldots, f\circ r_l)$ with $s \in \Clo(\algB)$ and $r_1,\ldots,r_l \in R_n(\algA)$ be a representation of $I$. Let $e \in \End(R_n^{k,k}(\algA))$. Then 
\begin{align*}
I(f)(e(\va)) &= s (f(r_1(e(\va))), \ldots, f(r_l(e(\va))))\\
&= s (f(e(r_1(\va))), \ldots, f(e(r_l(\va))))\\
&= I(f\circ e)(\va) = f(e(\va)).
\end{align*}
\end{proof}

In the case that $\algB$ is a module, we can prove some additional results. We are going to use the following notation:

\begin{definition}
Let $\algB$ be a module, and $A$ a arbitrary set. For $\va \in A^k$, and  $b \in B$, we define the map $\delta_\va^b \colon A^k \to B$ by 
$$
\delta_\va^b(\vx) = \begin{cases} b & \text{if } \vx = \va \\ 0 &\text{else.} \end{cases}
$$
If $\algB$ is a regular module over a ring with unity $1$, we will write $\delta_\va = \delta_\va^1$ for short.
\end{definition}

\begin{definition}
For a module $\algB$, we define the \emph{support} $\supp(f)$ of a function $f\colon A^n\to B$ as $\supp(f) = \{\va \in A^n \mid f(\va) \neq 0\}$.
\end{definition}

The following lemma shows how (if the target algebra $\algB$ is a module) we can lift uniform generation of characteristic functions $\delta_\va$, to all functions, whose support contains only elements from the orbit of $\va$, under the action of $\Aut(R_n^{k,k}(\algA))$:

\begin{lemma} \label{lemma:orbitsupport}
Let $\algA$ be a finite algebra, $T \subseteq A^k$, and $\algR$ be a ring such that $\{\delta_\va \mid \va \in T\}$ is uniformly generated by $n$-ary $(\algA,\algR)$-minors (where $\algR$ is considered as regular module).

Then, for $T' = \{ e(\va) \mid e \in \Aut(R_n^{k,k}(\algA)), \va \in T \}$ and every $\algR$-module $\algB$, also the set $\{ f\colon A^k\to B \mid \supp(f) \subseteq T' \}$ is uniformly generated by $n$-ary $(\algA,\algB)$-minors.
\end{lemma}

\begin{proof}
Let $I$ be an $(n,\algA,\algR)$-UR operation, such that $I(\delta_\va) = \delta_\va$, for all $\va \in T$, and let $\algB$ be a general $\algR$-module. Then $I$ can be represented by a formula $I(f)(\vx) = \sum_{s \in R_n^{k,k}(\algA)} r_s f(s(\vx))$, where every $r_s \in R$. Let $\algB$ be an arbitrary $\algR$-module. The formula for $I$ also induces operation $I \colon \clodO_{A,B}^{(k)} \to \clodO_{A,B}^{(k)}$, which, for simplicity, we will denote by the same letter. It is straightforward to see that $I(\delta_\va) = \delta_\va$ implies

$$I(\delta_\va^{b})(\vx) = \sum_{s \in R_n^{k,k}(\algA)} r_s \delta_\va^{b}(s(\vx)) = \sum_{s \in R_n^{k,k}(\algA)} r_s \delta_\va(s(\vx)) \cdot b = I(\delta_\va)\cdot b = \delta_\va\cdot b = \delta_{\va}^b.$$

By Lemma \ref{lemma:endomorphisms}, also $I(\delta_{e^{-1}(\va)}^b)(\vx) = I(\delta_\va^b)(e(\vx)) = \delta_\va^b (e(\vx)) = \delta_{e^{-1}(\va)}^b(\vx)$ holds, for all $e \in \Aut(R_n^{k,k}(\algA))$. Thus $I(\delta_\va^b) = \delta_\va^b$, for all $\va \in T'$, $b\in B$. Next, note that every function $f$ with support in $T'$ can be written as the sum $f = \sum_{\va \in T'} \delta_{\va}^{f(\va)}$. Since $\algB$ is a module, $I$ is linear, i.e., $I(f+g) = I(f) + I(g)$. Thus, we get that $I(f) = \sum_{\va \in T'} I(\delta_{\va}^{f(\va)}) = \sum_{\va \in T'} \delta_{\va}^{f(\va)} = f$, which finishes the proof.
\end{proof}

Lemma \ref{lemma:orbitsupport} implies two important facts: First, in order to prove that all operations from some algebra $\algA$ to a target $\algR$-module $\algB$ are uniformly generated by $n$-ary minors, it suffices to prove this for characteristic functions $\delta_\va$, such that the orbits of the tuples $\va$ under $\Aut(R_n^{k,k}(\algA))$ cover all of $A^k$.
Secondly, it is actually enough to only consider the regular module $\algR$ as target, in order to prove uniform generation for \emph{all} $\algR$-modules $\algB$:

\begin{corollary} \label{corollary:orbitsupport}
Let $\algA$ be a finite algebra and $\algR$ be a ring, considered as regular module. If $\clodO_{\algA,\algR}$ is uniformly generated by $k$-ary minors, then so is $\clodO_{\algA,\algB}$, for every $\algR$-module $\algB$.
\end{corollary}

\section{The main result}\label{sec:mainres} 

In this section, we are going to show that the $(k+1)$-ary functions from $\algA$ to $\algB$ are uniformly generated by their $k$-ary $(\algA, \algB)$-minors, if $\algA = \F^k$ is a finite vector space, and $\algB$ is a module in which $\chars(\F)$ is invertible. Corollary \ref{corollary:unigen} then implies that every clonoid from $\algA$ to $\algB$ is generated by its $k$-ary part.

Using the notation we introduced in the preliminary section, recall that $R_n(\algA)$ consists of all operations $X \mapsto MX$, such that $M$ is a matrix over $F$ of rank less than or equal to $n$. Thus, our goal is to prove the existence of coordinates $\alpha_M \in R_\mathbf{B}$ for every $(k+1)\times (k+1)$-matrix $M$ 
over $\F$ such that for each 
$f\colon \F^{(k+1) \times k} \to B$ we have 
\[
\forall X\in \F^{(k+1) \times k}\colon f(X)=
\sum_{\rk(M) \leq k} \alpha_M f(MX).
\]

Our proof proceeds, by first using the results from Section \ref{sec:unifgen}, to reduce the problem to the question, whether we can generate a characteristic function $\delta_{X_0}$, for matrices of full rank, by their $k$-ary minors. We discuss this reduction in Subsection \ref{sect:induction}. Then, we provide a proof of this fact, in Subsection \ref{sect:charfun}. In Subsection \ref{sect:proofsummary} we summarize the proof and discuss some immediate consequences. In Subsection \ref{sect:clonoidlattice} we use our results to obtain a more explicit description of the lattice of $(\F^n,\algB)$-clonoids. 

\subsection{Reduction to functions of singleton support} \label{sect:induction}

We start with the following lemma that can be proved similarly to Lemma \ref{lemma:orbitsupport}:

\begin{lemma} \label{lemma:basicinterpolation}
Let $X_0 \in \F^{m \times k}$ be a matrix with $\rk(X_0) = n$, and assume that $J$ is a $(n,\F^k,\algB)$-UR operation, given by a formula
\begin{equation} \label{eq:basicint1}
J(f)(X) = \sum_{\Col{M}=\Col{X_0}} \alpha_M f(MX),
\end{equation}
such that $J(\delta_{X_0})(X) = \delta_{X_0}(X)$, for all $X$ of rank $\rk(X) \leq i$.

Then, for all $f$ with $\supp(f) \subseteq \{X \mid \rk(X) \geq n\}$ and every $X\in \F^{m \times k}$ with $\rk(X) \leq i$:
\begin{equation}
J(f)(X) = \begin{cases} f(X) &\text{if } \Col{X} = \Col{X_0},\\
0 & \text{else}. \end{cases}
\end{equation}
\end{lemma}

\begin{proof}
We first consider the behavior of $J$ on certain characteristic functions. If $Z \in F^{m \times k}$ is a matrix with $\Col{Z} = \Col{X_0}$, then there is an invertible matrix $T \in F^{k \times k}$, such that $ZT = X_0$, so $\delta_Z({X}) = \delta_{X_0}(XT)$. By Lemma \ref{lemma:orbitsupport}, $J_H(\delta_Z)(X) = \delta_Z(X)$ holds, for all $X$ with $\rk(X) \leq i$.

Next, let $Z \in F^{m \times k}$ be matrix with $\rk(Z) \geq n$ and $\Col{Z} \neq \Col{X_0}$. Note that this implies that $\Col{MX} \subseteq \Col{X_0} \neq \Col{Z}$, for all $X \in \F^{m\times k}$ and $M\in \F^{m\times m}$ with $\Col{M} = \Col{X_0}$. Since $\dim(\Col{MX}) \leq n \leq \dim(\Col{Z})$, this further implies $Z \neq MX$. As a consequence, $J(\delta_Z)(X) = 0$ holds for every $X\in \F^{m\times k}$.

Next, let $f \in \clodO_{\F^k,\algB}^{(m)}$ be an arbitrary function, whose support contains only matrices of rank $\geq n$. Hence, $f$ can be written as $f= \sum_{\rk(Z) \geq n} \delta_Z^{f(Z)}$, and therefore $J(f) = \sum_{\Col{Z} = \Col{X}} J(\delta_Z^{f(Z)})$. 

If we evaluate $J(f)$ at some matrix $X$ with $\Col{X} = \Col{X_0}$ we obtain, by the above, $J(f)(X) = \delta_X^{f(X)} = f(X)$. If we evaluate $J(f)(X)$ at some $X$ with $\rk(X) \leq i$ and $\Col{X} \neq \Col{X_0}$ we similarly get $J(f)(X) = 0$. This finishes the proof.
\end{proof}

While it is a relatively technical statement on its own, Lemma \ref{lemma:basicinterpolation} has several very useful consequences that we will use throughout the paper. If we, for instance, assume that the equation \eqref{eq:basicint1} holds for all matrices $X \in \F^{m \times k}$, we get the following consequence:

\begin{lemma} \label{lemma:thetainterpolation}
Let $X_0 \in \F^{m \times k}$ be a matrix with $\rk(X_0) = n$ and assume that $\delta_{X_0}$ is generated by its $n$-ary $(\F^k,\algR)$-minors. Then 
\begin{enumerate}
\item For every $H \leq \F^{m}$ with $\dim(H) = n$, there is a $(n,\F^k,\algB)$-UR operation $J_H$, such that for all $f$ with $\supp(f) \subseteq \{ X \in F^{m \times k} \mid \rk(X) \geq n \}$:
$$J_H(f)(X) = \begin{cases} f(X) &\text{ if } \Col{X} = H\\ 
0 &\text{ else.} \end{cases}$$
\item There exists a $(n,\F^k,\algB)$-UR operation $J_n$ such that for all $f$ with $\supp(f) \subseteq \{ X \mid \rk(X) \geq n \}$:
$$J_n(f)(X) = \begin{cases} f(X) &\text{ if } \rk(X) = n\\ 
0 &\text{ else.} \end{cases}$$
\end{enumerate}
\end{lemma}

\begin{proof}
We first prove (1) for $H = \Col{X_0}$. By our assumption, we know that there are $\alpha_M \in \algR$ such that
\begin{equation} \label{eq:rankinterpol1} \delta_{X_0}(X) = \sum_{\rk(M)\leq n} \alpha_M \delta_{X_0}(MX) = \sum_{\Col{M}=\Col{X_0}} \alpha_M \delta_{X_0}(MX).
\end{equation}
For the second equation, note that for all matrices $M$ of rank $\leq n$, $\Col{M}\neq\Col{X_0}$ implies that $\delta_{X_0}(MX) = 0$ for all $X \in \F^{m\times k}$. So, without loss of generality, we can assume that $\alpha_M=0$ if $\Col{M}\neq\Col{X_0}$. Then, we obtain the statement by applying Lemma \ref{lemma:basicinterpolation} to  
\begin{equation} \label{eq:JHdefinition}
J_H(f) = \sum_{\Col{M}=\Col{X_0}} \alpha_M f(MX).
\end{equation}
Next, let $H \leq \F^{m}$ be an arbitrary subspace of dimension $n$. Then, there is an invertible matrix $S \in F^{m\times m}$ such that $\Col{SX_0} = H$. Note that $\delta_{SX_0}(X) = \delta_{X_0}(S^{-1}X)$, hence \eqref{eq:rankinterpol1} implies that 
$$\delta_{SX_0}(X) = \sum_{\Col{M}=\Col{X_0}} \alpha_M \delta_{X_0}(MS^{-1}X) = \sum_{\Col{M}=\Col{X_0}} \alpha_M \delta_{SX_0}(SMS^{-1}X).$$
Again, by Lemma \ref{lemma:basicinterpolation} we can see that $J_{H}(f)(X) = \sum_{\Col{M}=\Col{X_0}} \alpha_M f(SMS^{-1}X)$ has the desired properties.

For (2) we define $J_n(f) = \sum_{H \leq \F^m, \dim(H) = n} J_H(f)$. For every function $f$, whose support contains only matrices of rank $n$ or higher, note that by (1) $J_n(f)(X) = 0$, if $\rk(X) > n$, and $J_n(f)(X) = J_{\Col{X}}(f)(X) = f(X)$, if $\rk(X) = n$. So, we are done.
\end{proof}

We remark that the conditions of Lemma \ref{lemma:thetainterpolation} are in particular satisfied, if, for a given rank decomposition $X_0 = AU$ we have that $\delta_{X_0} \in \langle \delta_U \rangle_{\F^k,\algR}$ holds. Then, there are coefficients $\alpha_M \in \algR$ such that

\begin{equation} \label{eq:charfromrankdec}
\delta_{X_0}(X) = \sum_{M \in F^{n \times m }} \alpha_M \delta_{U}(MX) = \sum_{M \in F^{n \times m}} \alpha_M \delta_{X_0}(AMX).\end{equation}

As $\rk(AM) \leq n$, for every matrix $M \in F^{n \times m}$, equation \eqref{eq:charfromrankdec} clearly shows that $\delta_{X_0}$ is generated by its $n$-ary $(\F^k,\algR)$-minors.

We next prove that it is further enough to only consider matrices $X_0 \in \F^{(j+1)\times j}$ of full rank, for every $j\in [k]$, in order to prove that $\clodO_{\F^k,\algR}$ is uniformly generated by $k$-ary minors:

\begin{lemma} \label{lemma:reductiontosingleton}
Let $\F$ be a finite field, $\algR$ be a ring, $\algB$ be a $\algR$-module, and $k \geq 0$. Assume that, for every $j \in [k]$, there is a matrix $X_0 \in \F^{(j+1)\times j}$ of rank $\rk(X_0) = j$ such that $ \delta_{X_0}$ is generated by its $j$-ary $(\F^j,\algR)$-minors. Then
\begin{enumerate}
\item $\clodO_{\F^k,\algB}$ is uniformly generated by $k$-ary minors.
\item For every $i\in [k]_0$ and $m \in \N$ there is a $(i,\F^k,\algB)$-UR operation $I_i \colon \clodO_{\F^k,\algB}^{(m)} \to \clodO_{\F^k,\algB}^{(m)}$, with $I_i(f)(X) = f(X)$, for all $X \in \F^{m\times k}$ with $\rk(X) \leq i$.
\end{enumerate}
\end{lemma}

\begin{proof}
We first prove (1) by induction on $k=0,1,2,3,\ldots$ For the base step $k=0$, the statement trivially holds by $f(X) = f(0)$. (Also, $k=1$ can be observed by a careful reading of \cite{fioravanti-clonoids}).

So let us consider the induction step $k-1\to k$. For short, let us write $l = k+1$. Our goal is to prove that $\clodO_{F^{k},B}^{(l)}$ is uniformly generated by $k$-ary minors. By the induction assumption, there is a $(k-1,\F^{k-1},\algB)$-UR operation 

\[J(f)(Y) = \sum_{\substack{M \in \F^{l\times l} \\
\rk(M) \leq k-1}} \alpha_M f(MY) \]

such that $J(f)(Y) = f(Y)$, for all $f\in \clodO_{F^{k-1},B}^{(l)}$ and $Y \in \F^{l\times(k-1)}$. Note that this representation formula also induces a $(k-1,\F^{k},\algB)$-UR operation on $g \in \clodO_{F^k,B}^{(l)}$. Next, let $X \in \F^{l \times k}$ be a matrix of rank $\leq k-1$; so it can be decomposed into $X = YU$, for $Y\in \F^{l \times (k-1)}$ and $U\in \F^{(k-1) \times k}$. If, for a function $g \in \clodO_{F^k,B}^{(l)}$ we define $g_U \colon F^{l \times (k-1)} \rightarrow B$ by $g_U(Y) = g(YU)$, then we get 

$$J(g)(X) =  J(g)(YU) = J(g_U)(Y) = g_U(Y) = g(YU) = g(X).$$

Thus, $J(g)$ agrees with $g$ on matrices of rank $\leq k-1$. Next, recall that we assume that there is a $X_0 \in \F^{l\times k}$ such that $\delta_{X_0}$ is generated by $k$-ary minors. Lemma \ref{lemma:thetainterpolation} (2) implies, that there is an operation $J_k$ on $\clodO_{F^k,B}^{(l)}$ that can be represented by $k$-ary minors, such that $J_k(g) = g$, for all functions $g \in \clodO_{F^k,B}^{(l)}$ with $\supp(g) \subseteq \{X \mid \rk(X) = k\}$. It follows that, for every $g \in \clodO_{F^k,B}^{(l)}$:

$$g = J(g) + J_k(g-J(g)) = J(g) + J_k(g) - J_kJ(g).$$

Since both $J_k$ and $J$ can be represented by $k$-ary minors, we are done.

We next prove (2). For $m \leq i$, the statement is trivially true, by setting $I_i(f) = f$. So let us assume that $m>i$. By (1), we know that $\clodO_{\F^i,\algB}^{(m)}$ is uniformly generated by $i$-ary minors, so there exists a representation formula 
\[I_i(f)(Y) = \sum_{\substack{M \in \F^{m \times m} \\
\rk(M) \leq i}} \alpha_M f(MY), \]
such that $I_i(f)(Y) = f(Y)$ for all $Y \in \F^{m \times i}$. Naturally $I_i$ also induces an operation on $\clodO_{\F^k,\algB}^{(m)}$ that is representable by $i$-ary minors.

Every matrix $X \in \F^{m \times k}$ of rank at most $i$ can be written as $X = YU$, for $Y \in \F^{m \times i}, U \in \F^{i \times k}$. If we again, for every $g \in \clodO_{\F^k,\algB}^{(m)}$ define $g_U \colon F^{m \times i} \rightarrow B$ by $g_U(Y) = g(YU)$, we get 

$$I_i(g)(X) =  I_i(g)(YU) = I_i(g_U)(Y) = g_U(Y) = g(YU) = f(X).$$

Thus $I_i$ indeed satisfies $I_i(f)(X) = f(X)$, for all $X \in \F^{m\times k}$ with $\rk(X) \leq i$.
\end{proof}




\subsection{Generating characteristic functions} \label{sect:charfun}

By Lemma \ref{lemma:reductiontosingleton}, all there is left to prove our main result, is to show that for every $k\in \N$, there a matrix $X_0 \in \F^{(k+1)\times k}$ of rank $k$ such that $\delta_{X_0} \in \langle \delta_{\Id_k} \rangle_{\F^k,\algR}$, whenever $|F|$ is invertible in $\algR$.


So, throughout this subsection, we fix the following notation:
\begin{itemize}
    \item $\F$ denotes a fixed finite field
    \item $k \in \N$ is a fixed natural number
    \item $\algR$ is a regular module, in which $|F|$ is invertible.
    \item $X_0\in \F^{(k+1)\times k}$ is a fixed matrix of rank $k$, e.g.
    $$X_0 = \begin{bmatrix} \Id_k\\ 0
 \end{bmatrix} \in \F^{(k+1)\times k}.$$
\end{itemize}

Our goal is to show that there are coefficients $\alpha_M \in R$ for every $M \in F^{(k+1)\times k}$ of rank $\rk(M) = k$ such that $\delta_{X_0}(X) = \sum_{\rk(M) = k} \alpha_M \delta_{\Id_k}(MX)$. (Note that we only need to consider coefficient matrices $M \in \F^{k \times (k+1)}$ of full rank $k$, as otherwise $\delta_{\Id_k}(MX)$ is equal to $0$.) The function $X \mapsto \delta_{\Id_k}(MX)$ evaluates to $1$ only on the solution set of the equation $MX = \Id_k$. This motivates the following definition and lemma:

\begin{definition}
For every matrix $X \in \F^{(k+1) \times k}$ with $\rk(X) = k$ and for every vector $\va \in \F^{k+1} \setminus \Col{X}$, we define $\theta(X,\va) = \{ X + \va \vu^T \mid \vu \in \F^k \}$. We call every set of the form $\theta(X,\va)$ a \emph{$\theta$-space}.
\end{definition}

\begin{lemma} \label{lemma:thetaspaces} 
\begin{enumerate}
\item\label{ite:lemma:thetaspaces1} Let $M \in \F^{k\times (k+1)}$ with $\rk(M) = k$. Then there is a $\theta$-space $ V_M$ such that $MX = \Id_k$ if and only if $X \in V_M$. 
\item\label{ite:lemma:thetaspaces2} Conversely, for every $\theta$-space $V$, there is a matrix $M_V \in \F^{k\times (k+1)}$ with $\rk(M_V) = k$ such that $M_V X = \Id_k$ if and only if $X \in V$.
\end{enumerate}
\end{lemma}

\begin{proof}
We start with (\ref{ite:lemma:thetaspaces1}). 
Since $M$ has full rank, $\vx \mapsto M \vx$ is a surjective map from $\F^{k+1}$ to $\F^k$. Hence, there is a matrix $Y$ such that $MY = \Id_k$. Moreover, the kernel $\ker (M)$ has dimension 1, and is therefore generated by a single vector $\va \neq 0$. It follows that the solution space of the homogeneous equation $MY = 0$ consists of all matrices of the form $\va \vu^T$, for $\vu \in F^k$. 
Thus $ V_M = \theta(Y,\va)$ is the solution space to the system of linear equations $MX = \Id_k$. Furthermore, assuming $\va \in\Col{Y}$, that is, $\va = Y\vb$ for some $\vb \neq 0$, would imply that $0 = M\va = MY \vb = \Id_k\vb = \vb \neq 0$, a contradiction. Thus $V_M$ is indeed a $\theta$-space.

For (\ref{ite:lemma:thetaspaces2}), let $Y \in \F^{(k+1) \times k}$ and $\va \in \F^{(k+1)}\setminus\Col{Y}$ be such that $V = \theta(Y,\va)$. Let $\vn \in \F^{k+1}$ be a generator of $\ker(Y^T)$. Then, as in (\ref{ite:lemma:thetaspaces1}) we can see that there is a matrix $M_0 \in \F^{k \times (k+1)}$ such that $M Y = \Id_k$ holds if and only if $M = M_0 + \vm \vn^T$, for some $\vm \in \F^{k}$. Since $\va \notin \Col{Y}$, we have $ \vn^T\va \neq 0$. Then, let us define $\vm_V = - (\vn^T\va)^{-1}M_0\va$, and $M_V =  M_0 + \vm_V \vn^T$. By definition, $M_V \va = 0$, thus $M_V$ is indeed the matrix we were looking for.
\end{proof}

We are going to use the following notation:

\begin{definition} \label{definition:deltaV}
For a $\theta$-space $V$, let $\delta_V \colon \F^{k+1}\times k$ denote the function given by $\delta_V(X) = 1$ if $X \in V$ and $\delta_V(X) = 0$ else. Note that, by Lemma \ref{lemma:thetaspaces}
$\delta_{V}(X) = \delta_{\Id_k}(MX)$, for some matrix $M \in \F^{k\times (k+1)}$ of rank $k$.
\end{definition}

By definition, every $\theta$-space is an affine subspace of $\F^{(k+1)\times k}$ of dimension $k$. However, in a different sense, $\theta$-spaces 
exhibit some `line-like' behavior, by the following observations:

\begin{lemma} \label{lem:thetaspaces} 
Let $X_1, X_2\in \F^{(k+1) \times k}$, and let $\va_1, \va_2\in \F^{k+1}$ with $\va_1\notin \Col{X_1}$ and $\va_2\notin \Col{X_2}$.
\begin{enumerate}
\item \label{ite:lem:thetaspaces1} Either $|\theta(X_1,\va_1) \cap \theta(X_2,\va_2)| \leq 1$, or $\theta(X_1,\va_1) = \theta(X_2,\va_2)$.
\item \label{ite:lem:thetaspaces2} For $X_1 \neq X_2$ there is at most one $\theta$-space containing both $X_1$ and $X_2$.
\item \label{ite:lem:thetaspaces3} All elements of a $\theta$-space have rank $k$.
\end{enumerate}
\end{lemma}

\begin{proof}
For (\ref{ite:lem:thetaspaces1}), note that $\theta(X,\va) = \theta(Y,\va)$ holds for every $Y \in \theta(X,\va)$. Thus, 
if there exists $X$ in $\theta(X_1,\va_1) \cap \theta(X_2,\va_2)$, then without loss of generality we can assume $X = X_1 = X_2$. It is then easy to see that $\theta(X_1,\va_1) = \theta(X_2,\va_2)$ if $\va_1$ and $\va_2$ are collinear and $\theta(X_1,\va_1) \cap \theta(X_2,\va_2) = \{X\}$ else. 

For (\ref{ite:lem:thetaspaces2}), let $\va\in \F^{k+1}$. Then
$X_2 \in \theta(X_1,\va)$ iff $X_1-X_2 = \va \vu^T$ for some $\vu \neq 0$. 
In particular this can only happen if $\rk(X_1-X_2) = 1$. Therefore $\va \in \Col{X_1-X_2}$ is, up to multiplication with a scalar, uniquely determined by $X_1-X_2$.

To see (\ref{ite:lem:thetaspaces3}), let us consider a space 
$\theta(X,\va)$, and let $\vu\in \F^{k}$. We show that the columns of $X + \va \vu^T$ are linearly independent. 
To this end, let us consider a vector $\vb \in \F^{k}$ such that $(X + \va \vu^T)\vb = 0$. 
If $\vu^T \vb \neq 0$, then $\va = -X(\vu^T \vb)^{-1} \vb $, which contradicts the assumption that $\va\notin \Col{X}$. 
This leaves us with the case that $\vu^T \vb = 0$ and hence $X\vb = 0$. Since $\rk(X) = k$, this implies $\vb = 0$, and we are done.
\end{proof}

In the following definitions we introduce $\theta$-spaces of a specific form, that will play a crucial role in the proof of the main result of the section.

\begin{definition} \label{definition:type}
Let $n \leq k$, and let $A \in \F^{(k+1) \times n}$, $U \in \F^{n \times k}$ be matrices of rank $n$. Then, for $0 \leq j \leq n$, let us set $X_j := X_0 + \sum_{i=1}^j \va_i \vu_i^T$, 
where $\va_1,\ldots,\va_n \in \F^{k+1}$ denote the columns of $A$, and $\vu_1^T,\ldots,\vu_n^T \in \F^{k}$ denote the rows of $U$. 
We call the rank factorization $AU = X_n-X_0$ \emph{($n$-)admissible} if
\begin{enumerate}
\item $\va_{i+1} \notin 
\Col{X_i}$ for all $0 \leq i \leq n-1$;
\item $\va_{i+2} \in \Col{X_0} \cap \Col{X_1} \cap \cdots \cap \Col{X_i}$, for all $0 \leq i \leq n-2$.
\end{enumerate}
We say that $X \in \F^{(k+1)\times k}$ is of \emph{type $n$}, if $X = X_0 + AU$ for some $n$-admissible product $AU$, and we call $X= X_0 +AU$ an \emph{($n$-)admissible representation of $X$}. We write $\mathfrak T_n$ for the set of matrices of type $n$, and $\mathfrak T = \mathfrak T_0 \cup \mathfrak T_1 \cup \ldots \cup \mathfrak T_k$.
\end{definition}

\begin{definition} \label{definition:setoftype}
We say that a $\theta$-space $\theta(X,\va) = \{ X + \va \vu^T \mid \vu \in \F^k \}$ is of \emph{type $n$} if $X\in \mathfrak T_n$ with admissible representation $X = X_0 + AU$, and $\va \in \Col{X_0} \cap \Col{X_1}\ldots \cap \Col{X_{n-1}} \setminus \Col{X_n}$. 
We denote the set of all $\theta$-spaces of type $n$ by $\Theta_n$, and we let $\Theta = \Theta_0 \cup \Theta_1 \cup \cdots \cup \Theta_k$.
\end{definition}

The type of a matrix $X$ or the $\theta$-space $\theta(X,\va)$ does not depend on the particular choice of admissible representation $AU = X - X_0$. The proof requires the following lemma:

\begin{lemma} \label{lemma:representation}
Let $X \in \mathfrak T_n$, and 
let $X = X_0 + AU = X_0 + \tilde A \tilde U$ 
be two admissible representations of it. Furthermore, let $X_i = X_0 + \sum_{j=1}^i \va_j\vu_j^T$ and $\tilde X_i = X_0 + \sum_{j=1}^i \tilde \va_j \tilde \vu_j^T$ for $i \in[n]_0$. Then:
\begin{enumerate}
\item\label{ite:lemma:representation2} $\Col{X_0} \cap \Col{X_1} \cap \ldots \cap \Col{X_i} = \Col{\tilde X_0} \cap \Col{\tilde X_1} \cap \ldots \cap \Col{\tilde X_i}$ for 
every $i\in[n]_0$.
\item\label{ite:lemma:representation1} there is an lower triangular matrix $T \in \F^{n\times n}$ with $A = \tilde A T^{-1}$ and $U = T \tilde U$. 
\end{enumerate}
\end{lemma}

\begin{proof}
Since $AU = \tilde A \tilde U$ are two rank factorizations of $X-X_0$ we already know that there exists an invertible matrix $T= (t_{ij})_{i,j =1}^n$ with $A = \tilde AT$ and $U = T^{-1} \tilde U$.



For $i\in \{0,1,\ldots,n\}$, let $\vn_i, \tilde \vn_i$ denote vectors such that $\langle \vn_i \rangle = \Col{X_i}^\bot = \ker(X_i^T)$ and $\tilde \vn_i = \Col{\tilde X_i}^\bot = \ker(\tilde X_i^T)$. We are going to prove by induction on $i=0,1,2,\ldots, n$ that
\begin{enumerate}[label=(\alph*)]
\item\label{ite:lemma:representationa} $\langle \vn_0,\ldots,\vn_i \rangle = \langle \tilde \vn_0,\ldots,\tilde \vn_i \rangle$.
\item\label{ite:lemma:representationb} $\langle \vu_1,\ldots,\vu_i, \vu_{i+1} \rangle = \langle \tilde \vu_1,\ldots,\tilde \vu_{i}, \tilde \vu_{i+1} \rangle$.
\end{enumerate}

By considering the orthogonal spaces, we can see that \ref{ite:lemma:representationa} is equivalent to $\Col{X_0} \cap \Col{X_1} \cap \ldots \cap \Col{X_i} = \Col{\tilde X_0} \cap \Col{\tilde X_1} \cap \ldots \cap \Col{\tilde X_i}$. Furthermore, if $\langle \vu_1,\ldots, \vu_{l} \rangle = \langle \tilde \vu_1,\ldots,\tilde \vu_{l} \rangle$ holds for every $l\leq i+1$, then $TU = \tilde U$ implies that $t_{lj} = 0$ for all $l,j$ satisfying $i+1 \geq l>j$; in other words, $T$ looks like a lower triangular matrix on its first $i+1$ rows. So, if we prove (a) and (b) for $i=0,1,\ldots,n$, we are done.

We first prove \ref{ite:lemma:representationa} and \ref{ite:lemma:representationb} for $i=0$. By $X_0 = \tilde X_0$ clearly $\langle \vn_0 \rangle = \langle \tilde \vn_0 \rangle$ holds. For \ref{ite:lemma:representationb}, observe that $\vn_0^T\va_i \neq 0 \Leftrightarrow \va_i \not\in \Col{X_0} \Leftrightarrow i = 1$. This implies that $\vn_0^TX = \vn_0(X_0 + \sum_{j=1}^n \va_j\vu_j^T) = \vn_0^T\va_1\vu_1^T$. Symmetrically, we obtain $\vn_0^TX =\vn_0^T\tilde \va_1\tilde \vu_1^T$. Thus $\tilde \vu_1 = (\vn_0^T\va_1)^{-1}(\vn_0^T\va_1)\vu_1$, which finishes the proof of the base step.

Next, let us prove the induction step $1,\ldots,i-1\to i$. We start with \ref{ite:lemma:representationa}. By induction assumption $\langle \vn_0,\ldots,\vn_{i-1} \rangle = \langle \tilde \vn_0,\ldots,\tilde \vn_{i-1} \rangle$, thus, it is enough to show that $\tilde \vn_{i} \in \langle \vn_0,\ldots,\vn_{i-1}, \vn_i \rangle$. This is equivalent to finding coefficients $r_0,\ldots,r_{i-1}$, such that the equation $(\vn_i + \sum_{j=0}^{i-1}r_j\vn_j)^T \tilde X_i = 0$ holds.

By the induction assumption we know that $T$ has the shape of lower triangular matrix on its first $i$ rows. Recall that $\tilde A = AT$, and so $\tilde \va_j = \sum_{l\geq j} t_{lj} \va_l$ for every $j\leq i$. For $j < i$ this implies $\vn_j^T \tilde \va_{j+1} = t_{j+1,j+1} \vn_j^T \va_{j+1} \neq 0$. For all $j+1<l\leq i$ this further implies $\vn_j^T \tilde \va_{l} = 0$ holds. By both of these observations, for every $j < i$, we obtain 
\begin{equation} \label{eq:uis1} \vn_j^T \tilde X_i = \vn_j^T (\tilde X_j + \sum_{k>j} \tilde \va_{k}\tilde \vu_k^T) = c_j \vu_{j+1}^T, \text{ where } c_j = t_{j+1,j+1} \vn_j^T \va_{j+1} \neq 0.
\end{equation} On the other hand, note that
\begin{equation} \label{eq:uis2}
\vn_i^T \tilde X_i = \vn_i^T (\tilde X_i - X_i) = \vn_i^T (\sum_{j\leq i} \tilde \va_j \tilde\vu_j^T - \va_j\vu_j^T) \in \langle \vu_1^T, \ldots, \vu_j^T\rangle.
\end{equation}
The latter follows from $\langle \tilde\vu_1, \ldots, \tilde\vu_j\rangle = \langle \vu_1, \ldots, \vu_j\rangle$, which holds by the induction hypothesis. Thus, there are coefficients $d_1,\ldots, d_n$, such that $\vn_i^T \tilde X_i = \sum_{j\leq i} d_j \vu_j$. We conclude that $(\vn_i + \sum_{j=0}^{i-1}r_j \vn_j)^T \tilde X_i = 0$ is equivalent to the equation 
\begin{equation} \label{eq:uis3}
\sum_{j = 1}^i d_j \vu_j^T + \sum_{j=0}^{i-1} r_jc_j \vu_{j+1}^T = \sum_{j=0}^{i-1} (d_{j-1}+r_jc_j) \vu_{j+1}^T= 0 
\end{equation}
Equation \eqref{eq:uis3} has the solution $r_j = -d_{j-1}c_j^{-1}$ for $j=0,\ldots,i-1$ (recall that $c_j$ is invertible by \eqref{eq:uis1}). Thus $\tilde \vn_{i} \in \langle \vn_0,\ldots,\vn_{i-1}, \vn_i \rangle$, and we are done.

In order to prove \ref{ite:lemma:representationb} we are again going to use that we can write $\tilde \vn_i$ as linear combination $\tilde \vn_i = \sum_{j = 0}^{i} r_j \vn_j$ (with $r_i = 1$). By multiplying with $X$ we obtain $\tilde \vn_i^T X = \sum_{j=0}^i r_j \vn_j^T X = \sum_{j=0}^i r_j (\vn_j^T \va_{j+1})\vu_{j+1}^T$. On the other hand, clearly also $\tilde \vn_i^T X  = (\vn_i^T \tilde \va_{i+1}) \tilde \vu_{i+1}^T$ holds. Since $\vn_i^T \tilde \va_{i+1} \neq 0$, this implies that $\tilde \vu_{j+1} \in \langle \vu_{1}, \ldots \vu_{i+1}\rangle$. Since $\tilde \vu_{j+1} \notin \langle \tilde \vu_{1}, \ldots \tilde \vu_{i}\rangle = \langle \vu_{1}, \ldots \vu_{i}\rangle$, this implies that $\langle \tilde \vu_{1}, \ldots \tilde \vu_{i+1}\rangle = \langle \vu_{1}, \ldots \vu_{i+1}\rangle$, so we are done.

Since items \ref{ite:lemma:representationa} \ref{ite:lemma:representationb} hold for all $i=0,1,\ldots, n$, this finishes the proof of the lemma.
\end{proof}

By Lemma \ref{lemma:representation}(\ref{ite:lemma:representation2}), we know that the set 
$\Col{X_0} \cap \Col{X_1}\ldots \cap \Col{X_{n-1}} \setminus \Col{X_n}$  of admissible directions $\va$ in Definition \ref{definition:setoftype} does not depend on the concrete representation of $X = X_0 +AU$.

Furthermore, we get the following:

\begin{lemma}\label{lem:admrepA}
Let $X \in \mathfrak T_n$ with representation $X = X_0 + AU$, let $\vu_1^T,\ldots,\vu_n^T$ be the rows of $U$, and let $\theta(X,\va) \in \Theta_n$. Then 
\begin{enumerate}
\item\label{ite:lem:admrepA1} $X + \va \vu^T \in \mathfrak T_{n+1}$ if $\vu \notin \langle \vu_1, \ldots, \vu_n \rangle$;
\item\label{ite:lem:admrepA2} $X + \va \vu^T \in \mathfrak T_n$ if $\vu \in \langle \vu_1, \ldots, \vu_n \rangle$.
\end{enumerate}
\end{lemma}

\begin{proof}
\eqref{ite:lem:admrepA1} directly follows from the definition of $\mathfrak T_{n+1}$. For \eqref{ite:lem:admrepA2} note that, if $\vu = \sum_{i=1}^n r_i \vu_i$, then $X + \va \vu^T = X_0 + \sum_{i=1}^n (\va_i+r_i\va)\vu_i^T$, which is an $n$-admissible representation of $X + \va \vu^T$.
\end{proof}

From the previous lemma it follows that, if $\F$ is the field of order $q$, then $\theta$-spaces of type $n$ consist of $q^n$ matrices of type $n$, and $q^k - q^n$ matrices of type $n+1$. On the other hand, also the following holds:

\begin{lemma} \label{lemma:typecount}
Let $q = |\F|$ and let $X \in \mathfrak T_n$, for $1 \leq n \leq k$. Then 
\begin{enumerate}
\item\label{ite:lemma:typecount1} $X$ is in exactly one $\theta$-space of type $n-1$;
\item\label{ite:lemma:typecount2} $X$ is in exactly $q^{k-n}$ $\theta$-spaces of type $n$;
\item\label{ite:lemma:typecount3} $X$ is in no $\theta$-spaces of type $j \neq n-1,n$.
\end{enumerate}
\end{lemma}

\begin{proof}
Let $X = X_0 + AU$ be an admissible representation of $X$. So then $X = X_{n-1} + \va_{n} \vu_{n}^T$, with $X_{n-1} = X_0 + \sum_{i=1}^{n-1} \va_i \vu_i^T$. In particular $X \in \theta(X_{n-1},\va_n)$, which is a $\theta$-space of type $n-1$. Next, assume $X$ is in some other $\theta$-space of type $n-1$. This implies that there is a different admissible representation of $X = X_0 + \tilde A \tilde U$, with $\theta = \theta(\tilde X_{n-1},\tilde \va_n)$. By Lemma \ref{lemma:representation}, we know that representations are unique up to multiplication with lower triangle matrices $T$. In particular $\langle \tilde \va_n \rangle = \langle \va_n \rangle$. But this in turn implies that $\theta(X_{n-1},\va_n) = \theta(X,\va_n) = \theta(X,\tilde \va_n) = \theta(\tilde X_{n-1},\va_n)$, which completes the proof of \eqref{ite:lemma:typecount1}.

For \eqref{ite:lemma:typecount2} we simply need to count the number of admissible directions $\langle \va \rangle$ such that $\theta(X,\va)$ is a set of type $n$. Note that 
$\Col{X_0} \cap \Col{X_1}\cap \ldots \cap \Col{X_{n-1}}$ 
is a space of dimension ${k-n+1}$, thus 
$$|\Col{X_0} \cap \Col{X_1} \cap \ldots \cap \Col{X_{n-1}} \setminus \Col{X_n}| = q^{k-n+1} - q^{k-n} = q^{k-n}(q-1).$$ 
Hence, up to multiplication with some non-zero scalar, there are $q^{k-n}$ distinct choices for $\va$.

\eqref{ite:lemma:typecount3} clearly follows from the fact that every space of type $n$ contains only matrices of type $n$ or $n+1$, i.e., Lemma \ref{lem:admrepA}.
\end{proof}

We are now ready to prove the main result of this subsection:

\begin{theorem} \label{thm:reprdelta}
Let $k\in \N$, let $\F$ be a finite field of order $q$ and let $\algR$ be a regular module, in which $q$ is invertible. Let $X_0 \in \F^{(k+1)\times k}$ be of rank $\rk(X_0) = k$, and let $\Theta_n$, for $n\in [k]_0$ be as in Definition \ref{definition:type}. Then
$$\delta_{X_0}(X) = \sum_{n=0}^k (-1)^n q^{ \binom{n}{2} - (1 + n)k} \sum_{V \in \Theta_n} \delta_V(X).$$
Thus $\delta_{X_0}$ is in the $(\F^k,\algR)$-clonoid generated by $\delta_{\Id_k}$. This further implies that $\delta_{X_0}$ is generated by its $k$-ary $(\F^k,\algB)$-minors.
\end{theorem}

\begin{proof}
Assume that there are some coefficients $\alpha_0,\ldots, \alpha_k$, such that \[\delta_{X_0}(X) = \sum_{i=1}^k \alpha_i \sum_{V \in \Theta_i} \delta_V(X)\] holds for every $X \in \F^{k\times (k+1)}$.

If $X \notin \mathfrak T$, then this equation clearly holds, as all functions then evaluate to $0$. If $X = X_0$, i.e., $X$ has type $0$, we get the equation $1= q^{k} \alpha_0$. Evaluating at some matrix of type $i$ for $k\geq i>0$ gives us the equation $\alpha_{i-1} + q^{k-i} \alpha_i = 0$ by Lemma \ref{lemma:typecount}. By assumption $q$ is invertible in $\algR$, thus the equations are equivalent to $\alpha_0 = q^{-k}$ and $\alpha_i = -\alpha_{i-1}q^{i-k}$ for $i\in[k]$. It is now easy to see that this system of equations has the solution $\alpha_i = (-1)^i q^{\frac{i(i+1)}2 - (1+i)k}$. Since every $\delta_V$ is an $(\alg F^k, \alg R)$-minor of $\delta_{\Id_k}$ (see Definition \ref{definition:deltaV}), and $\delta_{\Id_k}(X) = \delta_{X_0}(X_0X)$, this finishes the proof.

\end{proof}

\subsection{Summary} \label{sect:proofsummary}

We now summarize the proof of our main result, and discuss some direct consequences.

\begin{theorem} \label{theorem:main}
Let $\F$ be a finite field, $k \in \N$, and let $\algB$ be an $\algR$-module, such that $|F|$ is invertible in $\algR$. Then $\clodO_{\F^k,\algB}$ is uniformly generated by $k$-ary $(\F^k,\algB)$-minors.
\end{theorem}

\begin{proof}
For every $j\in[k]$, let us pick some matrix $Z_j \in \F^{(j+1)\times j}$, such that $\rk(Z_j) = j$. Then, Theorem \ref{thm:reprdelta} implies that $\delta_{Z_j}$ is also generated by $j$-ary $(\F^j,\algR)$-minors. Lemma \ref{lemma:reductiontosingleton} (1) then implies that $\clodO_{\F^k,\algB}$ is uniformly generated by $k$-ary $(\F^k,\algB)$-minors, for every $\algR$-module $\algB$.
\end{proof}

As immediate corollaries of Theorem \ref{theorem:main} and Corollary \ref{corollary:unigen} we obtain the following:

\begin{corollary} \label{corollary:main}
Let $\F^k$ be the $k$-dimensional vector space over a finite field $\F$ and let $\algB$ be a faithful $\algR$-module, such that $|F|$ is invertible in $\algR$. Then every clonoid $\clodC$ from $\F^k$ to $\algB$ is generated by its $k$-ary part, i.e., $\clodC = \langle \clodC^{(k)}\rangle_{\F^k, \algB}$.
\end{corollary}

In particular, if $\algB$ is finite, we can prove the following result:

\begin{corollary} \label{corollary:main2}
Let $\F^k$ be the $k$-dimensional vector space over a finite field $\F$ and let $\algB$ be a finite module of coprime order. Then
\begin{enumerate}
\item there are only finitely many clonoids from $\F^k$ to $\algB$.
\item every clonoid from $\F^k$ to $\algB$ is finitely related by an at most $|F|^{k\times k}$-ary relation.
\end{enumerate}
\end{corollary}

\begin{proof}
Without loss of generality we can assume that $\algB$ is a faithful $\algR$-module. Then $|F|$ is invertible in $\algR$ if and only if $\algB$ is of coprime order.
The statement of the corollary then directly follows from Corollaries \ref{corollary:main}, \ref{corollary:unigen}, and \ref{corollary:fingen}.
\end{proof}

By combining our results with \cite{MW-clonoidsmodules} we are able to prove the following statement, that corresponds to the most general setting, in which Conjecture \ref{conjecture:main} is now resolved:

\begin{corollary} \label{corollary:mainresultproducts}
Let $\algA =\F_1^{k_1}\times \cdots \times \F_n^{k_n} \times \alg D$ as $\F_1\times \cdots \times\F_n \times \algR$-module, where all $\F_i$ are finite fields, and $\alg D$ is a finite distributive $\algR$-module. Let $l$ be the nilpotence degree of the Jacobson radical of $\algR$, and let $k = \max(k_1,\ldots,k_n,l)$. Further let $\algB$ be an $\algS$-module, such that $|A|$ is invertible in $\algS$. Then
\begin{enumerate}
\item every clonoid from $\alg A$ to $\algB$ is generated by its $k$-ary part.
\item If $\algB$ is finite then there are only finitely many clonoids from $\algA$ to $\algB$.
\end{enumerate}
\end{corollary}

\begin{proof}
Note that the clone of the source algebra factors over the clones of the single factors. Therefore, the stament follows from Theorem \ref{theorem:main}, Theorem \ref{theorem:distributive} and Proposition \ref{proposition:ugproduct}.
\end{proof}

In particular, Conjecture \ref{conjecture:main} holds for modules $\algA$ like in Corollary \ref{corollary:mainresultproducts}. For clonoids between polynomially equivalent algebras a similar statement for arity $k+1$ can be derived, by Lemma \ref{lem:unifgenforpolyequiv}.

\subsection{A description of the lattice of clonoids from $\F^k$ to $\algB$} \label{sect:clonoidlattice}

We next aim to refine Corollary \ref{corollary:main} to obtain an explicit description of the lattice of clonoids from $\F^k$ to $\algB$. By Corollary \ref{corollary:main}, this is equivalent to classify the submodules of $\algB^{F^{k \times k}}$ that are invariant under composition with matrices $F^{k\times k}$. We are going to do so in Theorem \ref{theorem:clonoidlattice}, using the following notation:

\begin{definition}
Let $\F$ a finite field and let $\algB$ be an $\algR$-module. Then 
\begin{itemize}
\item $\algR[\GL_i(\F)]$ is the group ring of $\GL_i(\F)$ over $\algR$. We denote its elements by formal sums $\sum_{M \in \GL_i(\F)} \alpha_M M$, for $\alpha_M \in R$.
\item For every $1 \leq i \leq k \in \N$, let us denote by $\algM_{i,k}(\F,\algB)$ the $\algR[\GL_i(\F)]$-module of all functions $f\colon \{X \in \F^{i \times k} \mid \rk(X) = i\} \to \algB$. The group ring naturally acts on the set by the action $$(\sum_{M \in \GL_i(\F)} \alpha_M M)f(X) = \sum_{M \in \GL_i(\F)} \alpha_M f(MX).$$
We further define $\algM_{0,k}(\F,\algB) = \algB$ (as $\algR$-module).
\end{itemize}
\end{definition}

To prove Theorem \ref{theorem:clonoidlattice},  we first prove the following characterization of $(\F^k,\algB)$-clonoids:

\begin{proposition} \label{proposition:grouprings}
Let $\F^k$ be the $k$-dimensional vector space over a finite field $\F$ and let $\algB$ be an $\algR$-module, such that $|F|$ is invertible in $\algR$.

Then, for all indices $0 \leq i\leq k \leq m$, and every matrix $A \in \F^{m \times i}$ of full rank $\rk(A) = i$, there is an $\algR$-linear embedding $L_A \colon \algM_{i,k}(\F,\algB) \to \clodO_{\F^k,\algB}^{(m)}$, such that 
\begin{enumerate}
\item $L_A(f)(AU) = f(U)$ if $\rk(U) = i$,
\item $L_A(f)(X) = 0$ if $\rk(X) \leq i$ and $\Col{X} \neq \Col{A}$,
\item $L_{TA}(f)(TX) = L_A(f)(X)$ for all $T \in \GL_m(\F)$,
\item $L_A(f)(X)$ can be computed in time $O(m^3)$,
\item $\clodC$ is an $(\F^k,\algB)$-clonoid, if and only if for every $i \in [k]_0$, the pre-images $L_A^{-1}(\cloC)$ are $\algR[\GL_i(\F)]$-submodules of $\algM_{i,k}(\F,\algB)$ that only depends on $i=\rk(A)$.
\end{enumerate}
\end{proposition}

\begin{proof}
Recall that by Theorem \ref{thm:reprdelta}, and Lemma \ref{lemma:thetainterpolation} (2) there are operations $I_0,I_1,\ldots,I_k$ on $\clodO_{\F^k,\algB}^{(m)}$, such that $I_i$ is $(i,\F^k,\algB)$-UR and $I_i(f)(X) = X$ for all matrices of rank $\rk(X) \leq i$. In particular, this implies that the assumptions of Lemma \ref{lemma:basicinterpolation} are satisfied, for any matrix $X_0 \in \F^{m\times k}$ and $i=n=\rk(X_0)$. So, for every $H\leq \F^m$ of dimension $\dim(H) \leq k$, Lemma \ref{lemma:basicinterpolation} implies that there is an operation $J_H(f)(X) = \sum_{\Col{M}=H}\alpha_M f(MX)$, such that for all $f$ with support in $\{X \mid \rk(X) \geq i\}$ we have $J_{H}(f)(X) = f(X)$ if $\Col{X} = H$ and $J_H(f)(X) = 0$ if $\Col{X} \neq H$ and $\rk(X) \leq i$. Without loss of generality we can further assume that $J_{T(H)}(f)(X) = \sum_{\Col{M}=H}\alpha_M f(TMT^{-1}X)$, for every $T \in \GL_m(\F)$ (cf. the proof of Lemma \ref{lemma:thetainterpolation} (1)).

We next define the operations $L_A$. Note that the map $f \mapsto L_A'(f) = L_A(f)|_{\{X \mid \rk(X) \leq i\}}$ is already uniquely determined by (1) and (2), and is an $\algR$-linear map from $\algM_{i,k}(\F,\algB)$ to the module $\{f \colon { \{X \in \F^{m\times m} \mid \rk(X) \leq i\}} \to \algB \}$. By its definition $J_H(f)$ only depends on the restriction of $f$ to matrices $X$ with $\Col{X} \leq H$ (and therefore of rank $\leq \dim(H)$). Thus, $L_A(f) = J_{\Col{A}}(L_A'(f))$ is well-defined. It follows directly from this definition, that (1) and (2) hold. Property (3) follows from $L'_{TA}(f)(TX) = L_A'(f)(X)$, and $J_{T(\Col{A})}(f)(X) = \sum_{\Col{M}=A}\alpha_M f(TMT^{-1}X)$.

To see (4), note that (for fixed $k$, and $f \in \algM_{i,k}(\F,\algB)$), we can compute $L_A'(f)(X)$ by checking if $X$ has a rank factorization $X = AU$ (which can be done in time $O(m^3)$). To compute $L_H(f)(X) = \sum_{\Col{M}=H}\alpha_M L_A'(f)(MX)$, involves computing the matrix products $MX$ (which is again bounded by $O(m^3)$) and evaluating $L_A'(f)(MX)$, for every $M$ with $\Col{M} = H$. As the number of matrices $M\in \F^{m\times i}$ with $\Col{M} = H$ does not depend on $m$, this finishes the proof.

We next prove (4). Let $\clodC$ be a $(\F^k,\algB)$-clonoid, and $m\geq k$.

We first observe that every $f \in \clodC^{(m)}$ can be decomposed into the sum $f = \sum_{i=0}^k K_i(f)$, where $K_0(f) = I_0(f)$, and $K_i(f) = I_i(f-K_{i-1}(f))$. By definition, for every $i \in [k]_0$ we have that $K_i$ is $(i,\F^k,\algB)$-UR, $K_i(f)(X) = 0$ for $\rk(X) \leq i-1$. Let us define $\algC_i = \{K_i(f) \mid f \in \cloC^{(m)}\}$; clearly $\clodC^{(m)}$ (as a $\algR$-module) decomposes into the direct sum of $\algC_0,\ldots,\algC_k$. Note that the functions in $\algC_i$ are already determined by their values on rank $i$ matrices. In fact, $\algC_i$ can be characterized as the set of all $f \in \clodC^{(m)}$ that are generated by their $i$-ary minors, and satisfy $f(X) = 0$ for $\rk(X) < i$ (as, for such functions, $K_i(f) = f$).

We can furthermore decompose every $\algC_i$ into a direct sum of $\algR$-modules $\algC_H = \{J_H(f) \mid f \in \algC_i \}$, where $J_H$ is defined as above. By construction, every function $J_H(f) \in \algC_H$ is completely determined by its restrictions to matrices of column-space $H$. On the other hand, $\algC_H$ can also be characterized as all $f\in \cloC^{(m)}$ that are generated by their $i$-ary minors, and satisfy $f(X) = 0$ if $\Col{X} \neq H$ and $\rk(X) \leq i$. For any $A$ with $\Col{A} = H$, the image of $L_H$ only consists of such functions. Therefore $L_A^{-1}(\cloC^{(m)}) = L_A^{-1}(\algC_H)$, if $\Col{A} = H$.

Next, note that for any $T \in \GL_m(\F)$, the map $I_T \colon \clodO_{\F^m,\algB}^{(m)} \to  \clodO_{\F^m,\algB}^{(m)}$ defined by $I_T(f)(X) = f(TX)$ is an isomorphism from $\algC_H$ to $\algC_{T(H)}$. By property (3) we have $L_{TA}(f) = L_A(I_T(f))$, therefore $L_A^{-1}(\cloC^{(m)}) = L_{TA}^{-1}(\cloC^{(m)})$, for every $T\in \GL_m(\F)$. In other words $L_A^{-1}(\cloC^{(m)})$ only depends on $i = \rk(A)$. We can see that $L_A^{-1}(\cloC^{(m)})$ is even an $\algR[\GL_i(\F)]$-module, by only considering maps $T \in \GL_m(\F)$ that fixes $\Col{A}$ setwise. Then $I_T$ is an automorphism of $\algC_H$, such that $I_T(f)$ only depend on the restriction of $T$ to $H$. In other words $\algC_H$ is module over $\algR[\GL(H)] = \algR[\GL_i(\F)]$, and so is its isomorphic copy $L_A^{-1}(\algC_H)$.
 
Conversely, for every $i \in [k]$, let $\algD_i$ be an $\algR[\GL_i(\F)]$-submodule of $\algM_{i,k}(\F,\algB)$. Let $\algD$ be the linear closure of $\{L_A(\algD_i) \mid i \in [k], \rk(A) = i\}$. We then need to show that $\algD$ is the $m$-ary part of a clonoid. (Then, by Theorem \ref{theorem:main}, and $m>k$ we are done.)

We already know that $\algD$ is an $\algR$-module. Thus, we only need to prove that it is closed under composition with matrices from the inside. So let $f \in \algD_i$, $A\in \F^{m\times i}$ with $\rk(A) = i$ and $M\in \F^{m\times m}$. Then, we need to prove that also $L_A(f)(MX) \in \algD$. If $\rk(MA) <i$, then we get $L_H(f)(MX) = 0$, by definition of $L_H$, so we are done. Otherwise $\rk(MA)  = i$, which implies that there is a linear transformation $T\in \GL_m(\F)$, such that $MA = TA$. Then, by property (3), we get $L_A(f)(MX) = L_A(f)(TX) = L_{T^{-1}A}(f)(X)$, which is also in $\algD$. This finishes the proof.
\end{proof}

\begin{theorem} \label{theorem:clonoidlattice}
The lattice of clonoids from $\F^k$ to $\algB$ is isomorphic to $\prod_{i=0}^k \Sub(\algM_{i,k}(\F,\algB))$.

Moreover, for $j\leq k$, the sublattice $\prod_{i=0}^j \Sub(\algM_{i,k}(\F,\algB))$ corresponds to the clonoids consisting of functions that are generated by their $j$-ary minors.
\end{theorem}

\begin{proof}
Let $m = k$, and let $(A_H)_{H \leq \F^k}$ be a family of matrices such that $\Col{A_H} = H$, and $A_H \in \F^{k \times \dim(H)}$. By Proposition \ref{proposition:grouprings} (3), the map $L \colon \prod_{i=0}^j \Sub(\algM_{i,k}(\F,\algB)) \to \mathcal P(\clodO_{\F^k,\algB}^{(k)})$ given by $L(\algC_0, \algC_1,\ldots,\algC_K) = \sum_{H\leq \F^K} L_{A_H}(\algC_{\dim(H)})$, is a lattice isomorphism from $\prod_{i=0}^j \Sub(\algM_{i,k}(\F,\algB))$ to the lattice of $k$-ary parts of clonoids. By Theorem \ref{theorem:main}, we are done.

To prove the second statement, note that by the proof Proposition \ref{proposition:grouprings}, $\cloC$ is generated by $j$-ary minors if and only if $L_A^{-1}(\cloC) = \{0\}$ for some $A\in \F^{m\times i}$ if and only if $i>j$.
\end{proof}

We remark that, in the special case where $k = 1$, and $\algB$ is a field, this description of the clonoid lattice was already derived in \cite{fioravanti-clonoids} (in fact, an even more explicit description could be computed, as then $\algM_{1,1}(\F,\algB) = \algB[\F^*]$ is a cyclic module).

\section{Subpower Membership Problem} \label{sect:SMP}

In this section, we discuss the Subpower Membership Problem for 2-nilpotent Mal'cev algebras $\algA$. As it was pointed out in \cite[Chapter 7]{FM-commutators}, to every 2-nilpotent Mal'cev algebra with central series $\mathbf 1_A > \mu > \mathbf 0_A$, one can associate two affine algebras $\algU = \algA/\mu$ and $\algL = \mu /\Delta_{\mu,\mathbf 1_A}$, such that $\algA$ is isomorphic to a \emph{central extension} (also called \emph{wreath product}) $\algU \otimes \algL$, whose domain is $U\times L$, and whose basic operations are of the form

$$f^{\algU \otimes \algL}\left( \begin{bmatrix} u_1\\ l_1 \end{bmatrix}, \ldots, \begin{bmatrix} u_n\\ l_n \end{bmatrix} \right) = \begin{bmatrix} f^{\algU}(u_1,\ldots,u_n)\\ f^{\algL}(l_1,\ldots,l_n) + \hat f(u_1,\ldots,u_n) \end{bmatrix},$$
for some $\hat f \colon U^n \to L$. In fact, then also all term operations can be represented as above.

Based on this representation, we define its \emph{difference clonoid} $\mathrm{Diff}(\algU \otimes \algL)$, as the set of differences $\hat t - \hat s$, for terms $t,s$ such that $t^\algU = s^{\algU}$ and $t^\algL = s^{\algL}$. As it was shown in \cite{kompatscher-SMP2nil}, $\clodC = \mathrm{Diff}(\algU \otimes \algL)$ is indeed a clonoid from $\algU$ to $\algL$ extended by some constant $0$, and  the subpower membership problem $\SMP(\algU \otimes \algL)$ reduces in polynomial time to the following problem:\\

{\noindent $\mathrm{CompRep}(\clodC)$}\\
\texttt{Input:} tuples $\va_1,\ldots,\va_m\in A^n$\\
\texttt{Output:} a generating set of $\{f(\va_1,\ldots,\va_m) \mid f \in \clodC^{(m)} \} \leq \algB^n$.\\

Here $f(\va_1,\ldots,\va_n)$ is computed component-wise. We are going to prove the following:

\begin{proposition} \label{prop:SMP}
Let $\cloC$ be a clonoid between $\algA$ and $\algB$, such that $\algA$ is polynomially equivalent to a finite vector space $\F^k$, and $\algB$ is a coprime module. Then $\mathrm{CompRep}(\clodC)$ is solvable in polynomial time.
\end{proposition}

\begin{proof}
Let us first assume that $\algA = \F^k$. By Proposition \ref{proposition:grouprings} (5), we know that there are submodules $\algC_i \leq \algM_{i,k}(\F,\algB)$ for all $i \in [k]_0$, such that $\clodC^{(m)}$ is generated by $\{L_A(f) \mid A \in \F^{m\times i}, \rk(A) = i, f \in \algC_i \}$. In fact, if for every $H \leq \F^m$ of dimension $\dim(H)\leq k$, we denote by $A_H \in \F^{\dim{(H)}\times m}$ a matrix with $\Col{A_H} =H$, then Proposition \ref{proposition:grouprings} implies that $f \in \clodC^{(m)}$ can be uniquely written as $f = \sum_{H\leq \F^k, \dim(H) =i} L_{A_H}(f_H)$, for some $f_H \in \algC_i$.

Now let us consider an instance of $\mathrm{CompRep}(\clodC)$. Using our established matrix notation, for given input matrices $X_1,\ldots,X_n \in \F^{m\times k}$, the task is to compute a generating set of the module $\{(f(X_1),f(X_2),\ldots,f(X_n)) \mid f \in \clodC^{m}\} \leq \algB$.

Let $\mathfrak H = \{H\leq \F^m \mid H \leq \Col{X_i},$ for some $ i \in [n]\}$. Note that the size of $\mathfrak H$ is linear in $n$. By picking a basis for every such space $H \in \mathfrak H$, we can enumerate matrices $(A_H)_{H \in \mathfrak H}$ as above in time $O(|A|^k \cdot n) = O(n)$.

From Proposition \ref{proposition:grouprings} (1)-(3) (and the decomposition $f = \sum_{H\leq \F^k, \dim(H) =i} L_{A_H}(f_H)$, for $f\in \clodC^{(m)}$), it follows that for all $f,g\in \cloC^{(m)}$, the equation $f(X_i) = g(X_i)$ holds, if and only if $L_{A_H}(f_H) = L_{A_H}(g_H)$, for all $H$ with $H \subseteq \Col{X_i}$. Hence $\{(f(X_1),f(X_2),\ldots,f(X_n)) \mid f \in \clodC^{m}\} $ is generated by $$\{(g(X_1),g(X_2),\ldots,g(X_n)) \mid H \in \mathfrak H, g = L_{A_H}(g_H),  g_H \in \algM_{\dim(H),k}(\F,\algB) \}.$$ Computing $L_{A_H}(g_H)(X_i)$ can be done in time $O(m^3)$ by Proposition \ref{proposition:grouprings} (4). Therefore $\mathrm{CompRep}(\clodC)$ is solvable in polynomial time $O(nm^3)$

In the case that $\algA$ is only polynomially equivalent to $\F^k$, then we still know (cf. Lemma \ref{lem:unifgenforpolyequiv}), that it contains terms $x+^y z = x-y+z$, $-^y( x) = y-x+y$ and $(r^y)_{r \in \F}$ with $r^y(x) = rx + (1-r)y$, such that the algebra $\algA^{y} = (A,+^y,y,-^y,(r^y)_{r \in \F})$ is isomorphic to $\F^k$ for every substitution of $y$ by some constant. As pointed out in Lemma \ref{lem:unifgenforpolyequiv} (1), the uniform generation by $k$-ary minors of operations from $\F^k$ and $\algB$, lifts to uniform generation by $k+1$-ary minors of operations from $\algA$ to $\algB$. Similar statements apply to the uniformly representable operations in Lemma \ref{lemma:basicinterpolation}, and in the proof of Proposition \ref{proposition:grouprings}. We leave the proof to the reader. Following the very same argument as for $\algA=\F^k$, we get that $\mathrm{CompRep}(\clodC)$ is solvable in polynomial time $O(nm^3)$.
\end{proof}

\begin{theorem}
Let $\algA = \algU \otimes \algL$ be a finite 2-nilpotent Mal'cev algebra in finite language, such that $\algU$ is polynomially equivalent to a vector space. Then $\SMP(\algA)$ is solvable in polynomial time.
\end{theorem}

\begin{proof}
By \cite[Theorem 10]{kompatscher-SMP2nil}, $\SMP(\algU \otimes \algL)$ reduces to $\mathrm{CompRep}(\clodC)$ in polynomial time, where $\clodC = \mathrm{Diff}(\algU,\algL)$ is a clonoid from $\algU$ to $\algL$, extended by a constant $0$. In particular, $\clodC$ is a clonoid from $\algU$ to the abelian group $\algL' =(L,+,0,-)$. If $\algL'$ is of coprime order to $\algU$, we are done by Proposition \ref{prop:SMP}. Else, note that $\algL'$ can be directly decomposed into $\algL' = \algL'_1\times \algL'_2$, such that $\algL_1'$ is coprime to $\algU$, and $\algL_2'$ is of size $p^n$, where the prime $p$ is the characteristic of $\algU$. Furthermore, $\Clo(\algL')$ contains the unary maps $\pi_1$ and $\pi_2$ that project to $\algL_1'$ and $\algL_2'$ respectively. 

Thus, every $f \in \clodC$ can also be decomposed into $f = \pi_1 f + \pi_2 f$, where $\pi_1 f$ is from a $(\algU,\algL_1')$-clonoid $\clodC_1$, and $\pi_2 f$ is from a $(\algU,\algL_2')$-clonoid $\clodC_2$. We can solve $\mathrm{CompRep}(\clodC_1)$ in polynomial time by Proposition \ref{prop:SMP}. Moreover, it can be seen (similar to the proof that supernilpotent Mal'cev algebras have polynomial time solvable SMP in \cite{mayr-SMP}) that $\mathrm{CompRep}(\clodC_2)$ can be solved in polynomial time. More precisely, we can be derived from the fact $\algA$ has a finite signature that there is a constant $C \in \N$ such that every $f(x_1,\ldots,x_m) \in \clodC_2$ can be written as a sum of functions $f_I \in \clodC_2$ with $I \subseteq [m]$ and $|I|\leq C$, such that $f_I$ only depend on the coordinates in $I$ (see \cite[Proof of Theorem 6.1]{KKK-CEQV2nil}). This means that $\mathrm{CompRep}(\clodC_2)$ can be solved in time $O(nm^C)$, by evaluating all operations $f_I \in \cloC$ with essential coordinates $I$, $|I|\leq C$ at the input. We leave the details of this proof to the reader. It follows that $\SMP(\algA)$ is in P.
\end{proof}

\section{Sharp lower bounds for the full clonoid}\label{sec:lowerbound}

In this section, we provide an elementary lower bound on the number of generators of the full clonoid between finite modules in general. 

\begin{theorem} \label{theorem:lowerbounds}
Let $\algA$ be a finite $\algR_A$-module and $\algB$ be a finite $\algR_B$-module. Assume that the full clonoid $\clodO_{\algA,\algB}$ is generated by the $m$-ary functions, for some $m$. Then $m \geq \frac{\log|A|}{\log|R_A|}$. 
\end{theorem} 

\begin{proof}
Let $G$ be a generating set of $\algB$ of minimal cardinality. Then, for every $n \in \N$, let $G_n \subseteq \clodO_{\algA,\algB}^{(n)}$ consist of all functions of the form $\delta_{\va}^g(X)$ with $\delta_{\va}^g(\va) = g \in G$, and $\delta_{\va}^g(\vx) = 0$ else, with  $\va \in \algA^n$. Note that $G_n$ is a generating set of the $\algR_B$-module $\algB^{\algA^n}$ of minimal cardinality $c\cdot |A|^n$. 

Now, assume that for $n \geq m$, $\clodO_{\algA,\algB}^{(n)}$ is in the clonoid generated by $\clodO_{\algA,\algB}^{(m)}$ (or, equivalently, $G_m$), and that $\algR_A$ is a right module. Then, $\clodO_{\algA,\algB}^{(n)}$ must be equal to the linear closure of all functions $g(M(x_1,\ldots,x_n)^t)$ such that $M$ is a $m\times n$-matrix over $\algR_A$, with $g \in \clodO_{\algA,\algB}^{(m)}$. There are at most $|G_m||R_A|^{mn}$ many distinct such functions. By the minimality of $|G_n|$ we get 
$$|G_m||R_A|^{mn} = c|A|^m|R_A|^{mn} \geq c|A|^n = |G_n|.$$
Considering the logarithm of the second inequality we get:
\begin{equation}\label{eq:loglower}
 m \geq \frac{n \log |A|}{\log |A| + n \log |R_A|}.   
\end{equation}
Since \eqref{eq:loglower} holds for every $n \geq m$, we obtain $m \geq \frac{\log|A|}{\log|R_A|}$.
\end{proof}

We remark that for vector spaces $\algA$ and $\algB$, this was already pointed out in \cite{vanecek-thesis}, together with a proof that, in the coprime case, $k$-ary functions suffice to generate the full clonoid. (Note however that this can also be derived from Proposition \ref{proposition:grouprings} (5)).

\section{Conclusion}\label{sec:conclusion}

In this paper, we studied clonoids between finite modules. We proved that there are only finitely many clonoids from a finite vector space $\algA$ to a coprime module $\algB$, which is also in accordance to Conjecture \ref{conjecture:main}.

Our proof (and, as we argued in Section \ref{sec:unifgen}, also the results for distributive modules $\algA$ in \cite{MW-clonoidsmodules}) uses that the operations from $A$ to $B$ are uniformly generated by $(\algA,\algB)$-minors. It is likely to expect that the same method works for all cases in the scope of Conjecture \ref{conjecture:main}. Thus, we also conjecture that the following statement is true:

\begin{conjecture} \label{conjecture:strong}
Let $\algA$ be a finite abelian $p$-group, and let $\algR$ be a ring (considered as regular $\algR$-module), such that $p$ is invertible in $\algR$. Then, there is a $k\in \N$, such that $\mathcal O_{\algA,\algR}$ is uniformly generated by $k$-ary $(\algA,\algR)$-minors.
\end{conjecture}

Note that, by Proposition \ref{proposition:ugproduct}, Conjecture \ref{conjecture:strong} implies the same statement holds for all abelian groups $\algA$ such that $|A|$ is invertible in $\algR$. Any module has clearly an abelian group as a reduct, thus we moreover get the same statement for modules $\algA$. Last, by Corollary \ref{corollary:orbitsupport}, Conjecture \ref{conjecture:strong} further implies that $\clodO_{\algA,\algB}$ is uniformly generated by $k$-ary $(\algA,\algB)$-minors for some $k$, for \emph{every} $\algR$-module $\algB$. In particular, this means that Conjecture \ref{conjecture:strong} implies Conjecture \ref{conjecture:main}.

If we consider only abelian $p$-groups $\algA$, then Theorem \ref{theorem:main} covers all elementary abelian groups $\algA = \Z_p^n$, while \cite{MW-clonoidsmodules} covers all cyclic groups $\algA = \Z_{p^n}$. Therefore, the easiest open problem is the following:

\begin{question} \label{question:pgroup}
Does Conjecture \ref{conjecture:strong} hold for $\algA = \Z_{p^2}\times \Z_p$, where $p$ is a prime?
\end{question}

Any progress on this question (and Conjecture \ref{conjecture:main} in general) likely hinges on combining the uniformly representable operations discussed in \cite{MW-clonoidsmodules} and the underlying paper. Note that, Theorem \ref{theorem:lowerbounds} and \cite[Lemma 5.1]{MW-clonoidsmodules} both only give us a lower bound of $k\geq 2$ on the arity of the generating minors.

In the context of the subpower membership problem, the second author conjectured in \cite{kompatscher-SMP2nil}, that $\SMP(\algA)$ of all finite 2-nilpotent Mal'cev algebras $\algA$ is polynomial time solvable. Our results in Section \ref{sect:SMP} confirm this for the case, in which the quotient of $\algA/\mu$ with respect to some central congruence $\mu$ is polynomially equivalent to a vector space. By building on the results in \cite{MW-clonoidsmodules}, we hope to also positively answer it in the case where $\algA/\mu$ is polynomially equivalent to a distributive module. This leads to the question:

\begin{question}
Is $\mathrm{CompRep}(\clodC)$ in P for every $(\algA,\algB)$-clonoid $\cloC$, with $\algA$ polynomially equivalent to a finite distributive module, and $\algB$ to a module?
\end{question}

We did not discuss any questions about equational theories in this paper. However, as it was exemplifies in \cite{mayr-VLloop} and Section 3 of \cite{KompatscherMayr2026}, finite basedness results for 2-nilpotent algebras can be derived from finite axiomatizations of the corresponding difference clonoids. Thus, last we ask:

\begin{question} \label{question:finitebasis}
Let $\algA$ and $\algB$ be finite modules, and let $\algC = \langle F \rangle_{\algA,\algB}$ be a finitely generated clonoid. Does $\clodC$ (or, more precisely, the multisorted algebra consisting of $\algA,\algB$ and $F$) always have a finite equational basis?
\end{question}

Due to the examples in \cite{mayr-VLloop} and \cite{KompatscherMayr2026} it is natural to expect uniformly representable operations (or, more precisely the formulas defining them), to play also a central role in answering Question \ref{question:finitebasis}.

\bibliographystyle{alphaurl}
\bibliography{clonoids.bib}
\end{document}